\documentclass[final,onefignum,onetabnum,letterpaper]{siamonline190516}
\usepackage[scale=0.8]{geometry} 

\usepackage{lipsum}
\usepackage{amsfonts}
\usepackage{epstopdf}
\ifpdf
  \DeclareGraphicsExtensions{.eps,.pdf,.png,.jpg}
\else
  \DeclareGraphicsExtensions{.eps}
\fi

\usepackage{datetime}
\newdateformat{monthyeardate}{%
  \monthname[\THEMONTH] \THEDAY, \THEYEAR}

\usepackage{academicons}
\usepackage{xcolor}
\renewcommand{\orcid}[1]{\href{https://orcid.org/#1}{\textcolor[HTML]{A6CE39}{orcid.org/#1}}}

\usepackage{amsmath}
\allowdisplaybreaks
\usepackage{amssymb}
\usepackage{commath}
\usepackage{mathtools}
\usepackage{bbm}

\usepackage{color}
\usepackage{graphicx}
\usepackage[small]{caption}
\usepackage{subcaption}

\usepackage{relsize}
\usepackage{adjustbox}
\usepackage{algorithm}
\usepackage[noend]{algpseudocode}
\usepackage{booktabs}
\usepackage{tikz}

\usepackage{verbatim}

\usepackage{enumitem}
\setlist[enumerate]{leftmargin=.5in}
\setlist[itemize]{leftmargin=.5in}

\newsiamthm{problem}{Problem}
\newsiamremark{remark}{Remark}
\newsiamremark{hypothesis}{Hypothesis}
\crefname{hypothesis}{Hypothesis}{Hypotheses}
\newsiamremark{example}{Example}
\newsiamthm{claim}{Claim}
\newsiamthm{conjecture}{Conjecture}

\headers{Generalized USBP operators \& application to DG methods}%
{Glaubitz, Ranocha, Winters, Schlottke-Lakemper, \"Offner, and Gassner}

\title{Generalized upwind summation-by-parts operators and their application to nodal discontinuous Galerkin methods
\thanks{
\monthyeardate\today
\corresponding{Jan Glaubitz}
}}

\author{
{Jan Glaubitz\thanks{
Department of Aeronautics and Astronautics, Massachusetts Institute of Technology, USA
(\email{jan.glaubitz@liu.se}, \orcid{0000-0002-3434-5563})
} }\thanks{
Department of Mathematics, Link\"oping University, Sweden
(\email{andrew.ross.winters@liu.se}, \orcid{0000-0002-5902-1522})
}
\and
Hendrik Ranocha\thanks{
Institute of Mathematics, Johannes Gutenberg University Mainz, Germany
(\email{hendrik.ranocha@uni-mainz.de}, \orcid{0000-0002-3456-2277})
}
\and
Andrew R.\ Winters\footnotemark[3]
\and
Michael Schlottke-Lakemper\thanks{
High-Performance Scientific Computing, University of Augsburg, Germany \& High-Performance Computing Center Stuttgart, Germany
(\email{michael.schlottke-lakemper@uni-a.de}, \orcid{0000-0002-3195-2536})
}
\and
Philipp \"Offner\thanks{
Institute of Mathematics, TU Clausthal, Germany
(\email{mail@philippoeffner.de}, \orcid{0000-0002-1367-1917})
}
\and
Gregor Gassner\thanks{
Department of Mathematics and Computer Science \& Center for Data and Simulation Science, University of Cologne, Germany
(\email{ggassner@uni-koeln.de}, \orcid{0000-0002-1752-1158})
}
}

\usepackage{amsopn}

\DeclareMathOperator{\diag}{diag}

\newcommand{\scp}[2]{\left\langle{#1, #2}\right\rangle}

\renewcommand{\d}{\mathrm{d}}
\newcommand{\intd}{\, \mathrm{d}}
\newcommand{\N}{\mathbb{N}}

\newcommand{\R}{\mathbb{R}}

\newcommand{\fnum}{f^{\operatorname{num}}}

\newsavebox{\DelimiterBox}
\newlength{\DelimiterHeight}
\newlength{\DelimiterDepth}
\newsavebox{\ArgumentBox}
\newlength{\ArgumentHeight}
\newlength{\ArgumentDepth}
\newlength{\ResizedDelimiterHeight}
\newlength{\ResizedDelimiterDepth}

\ifluatex

\else

\fi

\usepackage{lineno}

\ifpdf
\hypersetup{
  pdftitle={Generalized USBP operators},
  pdfauthor={Glaubitz}
}
\fi

\usepackage{siunitx}

\begin{document}

\maketitle

\begin{abstract}
High-order numerical methods for conservation laws are highly sought after due to their potential efficiency.  
However, it is challenging to ensure their robustness, particularly for under-resolved flows.  
Baseline high-order methods often incorporate stabilization techniques that must be applied judiciously---sufficient to ensure simulation stability but restrained enough to prevent excessive dissipation and loss of resolution. 
Recent studies have demonstrated that combining upwind summation-by-parts (USBP) operators with flux vector splitting can increase the robustness of finite difference (FD) schemes without introducing excessive artificial dissipation. 
This work investigates whether the same approach can be applied to nodal discontinuous Galerkin (DG) methods.
To this end, we demonstrate the existence of USBP operators on arbitrary grid points and provide a straightforward procedure for their construction. 
Our discussion encompasses a broad class of USBP operators, not limited to equidistant grid points, and enables the development of novel USBP operators on Legendre--Gauss--Lobatto (LGL) points that are well-suited for nodal DG methods.
We then examine the robustness properties of the resulting DG-USBP methods for challenging examples of the compressible Euler equations, such as the Kelvin--Helmholtz instability.
Similar to high-order FD-USBP schemes, we find that combining flux vector splitting techniques with DG-USBP operators does not lead to excessive artificial dissipation. 
Furthermore, we find that combining lower-order DG-USBP operators on three LGL points with flux vector splitting indeed increases the robustness of nodal DG methods. 
However, we also observe that higher-order USBP operators offer less improvement in robustness for DG methods compared to FD schemes. 
We provide evidence that this can be attributed to USBP methods adding dissipation only to unresolved modes, as FD schemes typically have more unresolved modes than nodal DG methods.
\end{abstract}

\begin{keywords}
	Upwind summation-by-parts operators, conservation laws, flux vector splittings, nodal discontinuous Galerkin methods
\end{keywords}

\begin{AMS}
	65M12, 65M60, 65M70, 65D25
\end{AMS}

\begin{Code}
    \url{https://github.com/trixi-framework/paper-2024-generalized-upwind-sbp}
\end{Code}

\begin{DOI}
	\url{https://doi.org/10.1016/j.jcp.2025.113841}
\end{DOI}

\section{Introduction}
\label{sec:introduction}

The simulation of many phenomena in science and engineering relies on efficient and robust methods for numerically solving hyperbolic conservation laws. 
Entropy-based methods have become increasingly favored for developing robust high-order methods applicable to various applications.
Initially introduced for second-order methods in the seminal work \cite{tadmor1987numerical,tadmor2003entropy}, these schemes have since evolved significantly.
Expanding on these foundational studies, \cite{lefloch2002fully,fjordholm2012arbitrarily} introduced high-order extensions tailored to periodic domains.
Further advancements were made in \cite{fisher2013high}, where the authors developed the \emph{flux differencing} approach to entropy stability.
This approach, based on SBP operators and two-point numerical fluxes, naturally extends to high-order operators based on complex geometries.

SBP operators are instrumental in ensuring a discrete analog of integration by parts, thereby facilitating the extension of properties of numerical fluxes to high-order methods.
SBP operators were initially developed in the 1970s, specifically for FD methods \cite{kreiss1974finite,kreiss1977existence,scherer1977energy}.
More contemporary overviews and references are provided in \cite{svard2014review,fernandez2014review,chen2020review}.
While the corpus of literature on SBP-based methods is too extensive to enumerate fully, notable implementations span a range of schemes, including spectral element \cite{carpenter2014entropy}, DG \cite{gassner2013skew,gassner2016split,chen2017entropy}, continuous Galerkin \cite{abgrall2020analysis,abgrall2021analysis}, finite volume (FV) \cite{nordstrom2001finite,nordstrom2003finite}, flux reconstruction (FR) \cite{ranocha2016summation,offner2018stability,cicchino2022nonlinearly,cicchino2022provably,montoya2022unifying}, and radial basis function (RBF) \cite{glaubitz2022energy} methods.

Several studies have established that flux differencing schemes, particularly DG spectral element methods (DGSEMs), are effective for under-resolved flows, such as the Taylor--Green vortex, when constructed to ensure entropy stability and positivity preservation. 
These properties are critical, especially in real applications involving strong shocks and high Reynolds numbers. 
See \cite{gassner2016split,sjogreen2018high,klose2020assessing,manzanero2020entropy,manzanero2020entropy2,rojas2021robustness,parsani2021high,lodares2022entropy,ntoukas2022entropy,yan2023entropy,montoya2024efficient} and references therein.
Furthermore, see \cite{abgrall2023,coqueil2021,renac2019entropy,renac2021} for additional applications of the DGSEM approach to multicomponent and multiphase flows. 
Entropy-stable high-order methods, like those studied in \cite{manzanero2020entropy,parsani2021high,yan2023entropy,montoya2024efficient}, possessed enhanced robustness for under-resolved flows without any built-in positivity preservation considerations.
However, works such as \cite{gassner2022stability,ranocha2021preventing} have demonstrated linear stability issues for high-order entropy-stable methods.
Despite their non-linear stability, these schemes exhibit shortcomings in seemingly straightforward test cases, such as Euler equations with constant velocity and pressure---reducing them to a simple linear advection for the density.
In contrast, central schemes without entropy properties perform well in this case but are prone to failure in more challenging simulations of under-resolved flows. 
Some resulting challenges in numerical simulations can be mitigated using shock-capturing or invariant domain-preserving techniques.
For example, issues such as non-negativity and spurious entropy growth can be effectively addressed by implementing positivity-preserving methods \cite{zhang2011maximum,qin2018implicit} and artificial viscosity approaches \cite{persson2006sub,ranocha2018stability}, respectively. 
However, it is crucial to apply these stabilization techniques judiciously---sufficient to ensure simulation stability but restrained enough to prevent excessive dissipation and loss of resolution. 

In this work, we are concerned with developing robust and efficient baseline schemes by employing USBP operators within a nodal DG framework---without focusing on positivity preservation, however.
Introduced in \cite{mattsson2017diagonal} for FD methods, USBP operators emerged as a particular class of dual-pair SBP operators \cite{dovgilovich2015high}.
Henceforth, we refer to these operators as \emph{FD-USBP operators}.
These FD-USBP operators have been extended to staggered grids in \cite{mattsson2018compatible} and applied in the context of the shallow water equations and residual-based artificial viscosity methods in \cite{lundgren2020efficient,stiernstrom2021residual}, respectively. 
Furthermore, for instance, \cite{ortleb2023stability} demonstrated that DG discretizations of linear advection-diffusion equations can be analyzed using the framework of global USBP operators.
For nonlinear conservation laws, USBP operators are combined with flux vector splitting techniques \cite[Chapter 8]{toro2013riemann}.
Flux vector splitting techniques were extensively developed and applied throughout the previous century \cite{steger1979flux,vanleer1982flux,buning1982solution,hanel1987accuracy,liou1991high,coirier1991numerical} yet eventually fell out of favor.
This shift occurred as these methods were superseded by alternative approaches, primarily due to the considerable numerical dissipation they introduced \cite{vanleer1991flux}.
Yet, previous studies \cite{ranocha2023high,duru2024dual} have demonstrated that combining high-order USBP operators with flux vector splittings can increase the robustness of FD schemes for under-resolved simulations without introducing excessive artificial dissipation.

\subsection*{Our contribution}

In this work, we investigate whether combining USBP operators with flux vector splittings can be applied to nodal DG methods with similar success. 
To this end, we first demonstrate the existence of USBP operators on arbitrary grid points and provide a straightforward procedure for their construction. 
While previous works have developed diagonal-norm FD-USBP operators, our discussion encompasses a broader class of USBP operators.
Importantly, the proposed construction procedure applies to arbitrary grid points, not limited to equidistant spacing.
This generalization enables the development of novel USBP operators on LGL points that are well-suited for nodal DG methods.
In contrast to, for instance, \cite{ortleb2023stability}, we not only use a global upwind formulation of DG methods but also develop local USBP operators for nodal DG-type methods. 
These operators achieve a given order of accuracy with fewer degrees of freedom (nodes) than existing FD-USBP operators \cite{mattsson2017diagonal,mattsson2018compatible}.
We then examine the robustness properties of the resulting DG-USBP methods for challenging examples of the compressible Euler equations, such as Kelvin-Helmholtz instabilities.
Similar to high-order FD-USBP schemes, we find that combining flux vector splitting techniques with DG-USBP operators does not lead to excessive artificial dissipation. 
Furthermore, we demonstrate local linear/energy stability for the DG-USBP semi-discretizations in a setting where existing high-order entropy-stable flux differencing DGSEM produce unstable semi-discretizations. 
At the same time, we observe that USBP operators offer less improvement in robustness for high-order DG methods than for FD schemes for multi-dimensional systems. 
Specifically, we note in our numerical simulations that the increase in robustness becomes less significant as more LGL points are used per element and coordinate direction. 
We suspect that this reduced improvement can be attributed to USBP methods adding dissipation only to unresolved modes, and FD schemes typically have a larger number of unresolved modes than nodal DG methods. 
Still, DG-USBP methods can serve as a viable alternative approach, as it may be more straightforward to apply in complex problems where a robust flux splitting, such as Lax-Friedrichs, is readily available while conducting an entropy analysis proves too cumbersome.

\subsection*{Outline}

In \Cref{sec:USBP}, we present all necessary background on SBP and USBP operators.
\Cref{sec:existence} is devoted to delineating the conditions for the existence of USBP operators.
\Cref{sec:construction} outlines a systematic approach for constructing USBP operators.
In \Cref{sec:application}, we use the previously developed DG-USBP operators to formulate DG-USBP baseline schemes.
\Cref{sec:tests} is reserved for various numerical experiments.
Finally, \Cref{sec:summary} offers concluding remarks and explores potential avenues for future research.

\section{Preliminaries on SBP and USBP operators}
\label{sec:USBP}

We start by revisiting SBP and USBP operators.
For simplicity, we focus on one-dimensional operators on a compact interval.

\subsection{Notation}
\label{sub:notation}

Consider $\Omega = [x_{\min}, x_{\max}]$ as the computational spatial domain for solving a given time-dependent partial differential equation (PDE).
Our focus is on SBP operators, which apply to a generic element $\Omega_{\text{ref}} = [x_L, x_R]$.
In pseudo-spectral or global FD schemes, $\Omega_{\text{ref}} = \Omega$.
However, in DG and multi-block FD methods, $\Omega$ is divided into multiple smaller non-overlapping elements $\Omega_j$ which collectively cover $\Omega$ (i.e., $\bigcup_{j} \Omega_j = \Omega$).
In this case, any $\Omega_j$ can serve as $\Omega_{\text{ref}}$.
Let $\mathbf{x} = [x_1,\dots,x_N]^T$ represent a vector of grid points on the reference element $\Omega_{\rm ref}$, where $x_L \leq x_1 < \dots < x_N \leq x_R$.
For a continuously differentiable function $f$ on $\Omega_{\rm ref}$, we denote
\begin{equation}
\mathbf{f} = [ f(x_1), \dots, f(x_N) ]^T, \quad
\mathbf{f'} = [ f'(x_1), \dots, f'(x_N) ]^T,
\end{equation}
as the nodal values and first derivative of $f$ at the grid points $\mathbf{x}$.
Additionally, $\mathcal{P}_d$ denotes the space of polynomials of degree up to $d$.

\subsection{SBP operators}
\label{sub:SBP}

We are now positioned to define SBP operators.
See the reviews \cite{svard2014review,fernandez2014review,chen2020review} and references therein for more details.

\begin{definition}[SBP operators]\label{def:SBP}
	The operator $D = P^{-1}( Q + B/2 )$ approximating $\partial_x$ is a \emph{degree $d$ SBP operator} on $\Omega_{\text{ref}} = [x_L,x_R]$ if
	\begin{enumerate}
		\item[(i)]
		$D \mathbf{f} = \mathbf{f\,'}$ for all $f \in \mathcal{P}_d$;

		\item[(ii)]
		$P$ is a symmetric and positive definite (SPD) matrix;

		\item[(iii)]
		The boundary matrix $B$ satisfies $\mathbf{f}^T B \mathbf{g} = fg \big|_{x_L}^{x_R}$ for all $f,g \in \mathcal{P}_{d}$;

		\item[(iv)]
		$Q + Q^T = 0$ ($Q$ is anti-symmetric).

	\end{enumerate}
\end{definition}

Relation (i) ensures that $D$ accurately approximates the continuous derivative operator $\partial_x$ by requiring $D$ to be exact for polynomials up to degree $d$.
Condition (ii) guarantees that $P$ induces a proper discrete inner product and norm, which are given by $\scp{\mathbf{f}}{\mathbf{g}} = \mathbf{f}^T P \mathbf{g}$ and $\| \mathbf{f} \|_P^2 = \mathbf{f}^T P \mathbf{f}$, respectively.
Relation (iv) encodes the SBP property, which allows us to mimic integration-by-parts (IBP) on a discrete level.
Recall that IBP reads
\begin{equation}\label{eq:IBP}
 	\int_{x_L}^{x_R} f (\partial_x g) \intd x + \int_{x_L}^{x_R} (\partial_x f) g \intd x
		= fg \big|_{x_L}^{x_R}.
\end{equation}
The discrete version of \cref{eq:IBP}, which follows from (iv) in \cref{def:SBP}, is
\begin{equation}\label{eq:SBP}
	\mathbf{f}^T P D \mathbf{g} + \mathbf{f}^T D^T P^T \mathbf{g}
		= \mathbf{f}^T B \mathbf{g}.
\end{equation}
Finally, relation (iii) in \cref{def:SBP} ensures that the right-hand side of \cref{eq:SBP} accurately approximates the one of \cref{eq:IBP} by requiring the boundary operator $B$ to be exact for polynomials up to degree $d$ as well.

\subsection{USBP operators}
\label{sub:USBP}

Next, we introduce USBP operators, which originate from the FD setting \cite{mattsson2017diagonal} and are a subclass of dual-pair SBP operators \cite{dovgilovich2015high}.

\begin{definition}[USBP operators]\label{def:USBP}
	The operators $D_+ = P^{-1}( Q_+ + B/2 )$ and $D_- = P^{-1}( Q_- + B/2 )$ approximating $\partial_x$ are \emph{degree $d$ USBP operators} on $\Omega_{\text{ref}} = [x_L,x_R]$ if
	\begin{enumerate}
		\item[(i)]
		$D_+ \mathbf{f} = D_- \mathbf{f} = \mathbf{f\,'}$ for all $f \in \mathcal{P}_d$;

		\item[(ii)]
		The norm matrix $P$ is an SPD matrix;

		\item[(iii)]
		The \emph{boundary matrix} $B$ satisfies $\mathbf{f}^T B \mathbf{g} = fg \big|_{x_L}^{x_R}$ for all $f,g \in \mathcal{P}_{\tau}$ with $\tau \geq d$;

		\item[(iv)]
		$Q_+ + Q_-^T = 0$;

		\item[(v)]
		$Q_+ + Q_+^T = S$, where the \emph{dissipation matrix} $S$ is symmetric and negative semi-definite.

	\end{enumerate}
\end{definition}

The motivation for conditions (i), (ii), and (iii) in \cref{def:USBP} is the same as in \cref{def:SBP} of SBP operators.
However, the discrete version of IBP \cref{eq:IBP} that follows from (iv) in \cref{def:USBP} is
\begin{equation}\label{eq:USBP}
	\mathbf{f}^T P D_+ \mathbf{g} + \mathbf{f}^T D_-^T P^T \mathbf{g}
		= \mathbf{f}^T B \mathbf{g}.
\end{equation}
At the same time, (v) introduces accurate artificial dissipation into the USBP discretization.
This can be seen as follows (also see \cite{mattsson2017diagonal,linders2020properties,linders2022eigenvalue}):
Using (iv) and (v), we observe that
\begin{equation}\label{eq:relation_D_pm}
	D_+ = D_- + P^{-1}S,
\end{equation}
and it follows that $D_+ \mathbf{f} = D_- \mathbf{f}$ if and only if $S \mathbf{f} = \mathbf{0}$.
This means that $S$ adds artificial dissipation that acts only on unresolved modes, $\mathbf{f} \in \R^N$ for which $f \not\in \mathcal{P}_d$.
The artificial dissipation included in USBP operators only acting on unresolved modes distinguishes them from combining SBP operators with traditional artificial dissipation terms based on second (or higher-order even) derivatives \cite{mattsson2004stable,ranocha2018stability}.
\cref{lem:connection} connects SBP and USBP operators.

\begin{lemma}[USBP operators induce SBP operators]\label{lem:connection}
	If $D_{\pm} = P^{-1}( Q_{\pm} + B/2 )$ are degree $d$ USBP operators, then
	\begin{equation}
		D := (D_+ + D_-)/2 = P^{-1}( Q + B/2 )
		\quad \text{with} \quad
		Q = (Q_+ + Q_-)/2
	\end{equation}
	is a degree $d$ SBP operator.
	Moreover, $D_+ - D_- = P^{-1} S$.
\end{lemma}

\begin{proof}
	\cref{lem:connection} was demonstrated in \cite{mattsson2017diagonal} and we thus omit its proof.
\end{proof}

We will need \cref{lem:connection} to prove \cref{thm:existence} below, which characterizes the existence of USBP operators.
\section{Characterizing the existence of USBP operators}
\label{sec:existence}

We now characterize the existence of USBP operators as in \cref{def:USBP} for arbitrary grid points.
We subsequently leverage this analysis in \Cref{sec:construction} to construct novel USBP operators.

\begin{theorem}\label{thm:existence}
	Let $\mathbf{x} = [x_1,\dots,x_N]^T$ be a grid on $\Omega_{\text{ref}} = [x_L,x_R]$ and let $S$ be a symmetric and negative semi-definite dissipation matrix.
	\begin{itemize}
		\item[(a)]
		If there exists a degree $d$ SBP operator $D = P^{-1}( Q + B/2 )$ and $S$ satisfies $S \mathbf{f} = \mathbf{0}$ for all $f \in \mathcal{P}_d$, then $D_{\pm} = P^{-1}( Q_{\pm} + B/2 )$ with $Q_{\pm} = Q \pm S/2$ are degree $d$ USBP operators with $Q_+ + Q_+^T = S$.

		\item[(b)]
		If there exist degree $d$ USBP operators $D_{\pm} = P^{-1}( Q_{\pm} + B/2 )$ with $Q_+ + Q_+^T = S$, then $D = (D_+ + D_-)/2$ is a degree $d$ SBP operator and $S$ satisfies $S \mathbf{f} = \mathbf{0}$ for all $f \in \mathcal{P}_d$.

	\end{itemize}
\end{theorem}

In a nutshell, \cref{thm:existence} tells us that there exist degree $d$ USBP operators $D_{\pm}$ on a given grid if and only if there exists a degree $d$ SBP operator $D$ on the same grid and a dissipation matrix $S$ with $S \mathbf{f} = \mathbf{0}$ for all $f \in \mathcal{P}_d$.
The proof of \cref{thm:existence} is provided in \cref{app:proof_existence}.

While the dissipation matrix $S$ should not add dissipation to resolved modes, i.e., $S \mathbf{f} = \mathbf{0}$ for all $f \in \mathcal{P}_d$, it should do so for \emph{unresolved} modes.
The latter means that
\begin{equation}\label{eq:dissip_cond}
	\mathbf{f}^T S \mathbf{f} < 0 \quad \forall f \not\in \mathcal{P}_d.
\end{equation}
We follow up with \cref{lem:S_char}, which shows that a desired dissipation matrix $S$ with $S \mathbf{f} = \mathbf{0}$ if $f \in \mathcal{P}_d$ and $\mathbf{f}^T S \mathbf{f} < 0$ otherwise always exists.

\begin{lemma}\label{lem:S_char}
	Let $d \in \N$, $\mathbf{x}$ be a grid with $N \geq d+1$ distinct points.
	There always exists a symmetric and negative semi-definite matrix $S$ with $S \mathbf{f} = \mathbf{0}$ if $f \in \mathcal{P}_d$ and $\mathbf{f}^T S \mathbf{f} < 0$ otherwise.
	Moreover, $S = 0$ if $N = d+1$ and $S \neq 0$ if $N > d+1$.
\end{lemma}

\begin{proof}
	If $N = d+1$, then any vector $\mathbf{f} \in \R^N$ corresponds to the nodal values of a function $f \in \mathcal{P}_d$ on the grid $\mathbf{x}$.
	There are no unresolved modes in this case, and $S = 0$ is the desired matrix.

	Now assume that $N > d+1$.
	Since the matrix $S$ is real and symmetric, it can be decomposed as
	\begin{equation}\label{eq:S_char_proof1}
		S = V \Lambda V^{-1},
	\end{equation}
	where $\Lambda = \diag(\lambda_1,\dots,\lambda_N)$ contains the eigenvalues of $S$ and the columns of $V = [\mathbf{v_1},\dots,\mathbf{v_N}]$ are linearly independent eigenvectors of $S$; see \cite[Theorem 8.1.1]{golub2012matrix}.
	Let $\{f_k\}_{k=1}^{d+1}$ be a basis of $\mathcal{P}_d$, e.g., monomials or Legendre polynomials up to degree $d$.
	Then, $S \mathbf{f} = \mathbf{0}$ if $f \in \mathcal{P}_d$ is equivalent to $S \mathbf{f_k}= \mathbf{0}$ for $k=1,\dots,d+1$.
	At the same time, the latter condition means that $\mathbf{f_1},\dots,\mathbf{f_{d+1}}$ are eigenvectors of $S$ with the corresponding eigenvalue being zero.
	For simplicity, we position these as the first $d+1$ eigenvalues and eigenvectors in \cref{eq:S_char_proof1}, i.e.,
	\begin{equation}
		\lambda_k = 0, \quad \mathbf{v_k} = \mathbf{f_k}, \quad k=1,\dots,d+1.
	\end{equation}
	Then, by construction, the matrix $S$ in \cref{eq:S_char_proof1} is symmetric and satisfies $S \mathbf{f} = \mathbf{0}$ if $f \in \mathcal{P}_d$.
	Furthermore, since $N > d+1$, we still have $N-d-1>0$ free eigenvalues that we can use to ensure that $\mathbf{f}^T S \mathbf{f} < 0$ if $f \not\in \mathcal{P}_d$.
	To this end, the remaining eigenvalues must be negative, i.e., $\lambda_{d+2},\dots,\lambda_N < 0$.
	Finally, we can set $\mathbf{v_{d+2}},\dots,\mathbf{v_{N}}$ as arbitrary vectors that are linearly independent among themselves and to the first $d+1$ eigenvectors.
	Then, the resulting matrix $S$ in \cref{eq:S_char_proof1} is symmetric and negative semi-definite while satisfying $S \mathbf{f} = \mathbf{0}$ if $f \in \mathcal{P}_d$.
	It remains to show that $\mathbf{f}^T S \mathbf{f} < 0$ if $f \not\in \mathcal{P}_d$.
	To this end, suppose that $f \not\in \mathcal{P}_d$.
	In this case, the corresponding vector of nodal values $\mathbf{f}$ can be expressed as $\mathbf{f} = \gamma_{1} \mathbf{v_{1}} + \dots + \gamma_{N} \mathbf{v_{N}}$ with $\gamma_k \neq 0$ for at least one $k > d+1$.
	(At least one of the last $N-d-1$ eigenvectors contributes to representing $\mathbf{f}$.)
	Then, we see that
	\begin{equation}
		\mathbf{f}^T S \mathbf{f}
			= \underbrace{\sum_{k=1}^{d+1} \overbrace{\lambda_k}^{= 0} \gamma_k^2 \mathbf{v_k}^T \mathbf{v_k}}_{= 0}
			+ \underbrace{\sum_{k=d+2}^{N} \overbrace{\lambda_k}^{< 0} \gamma_k^2 \mathbf{v_k}^T \mathbf{v_k}}_{< 0}
			< 0,
	\end{equation}
	which completes the proof.
\end{proof}

Many modern collocation-type DG methods are based on interpolatory ($N = d+1$) SBP operators. 
See \cite{gassner2013skew,gassner2016split,klose2020assessing,manzanero2020entropy2, chen2020review,parsani2021high,yan2023entropy} and references therein. 
However, \cref{lem:S_char} indicates that there \emph{cannot} be interpolatory USBP operators with a non-trivial dissipation matrix $S$. Conversely, \cref{lem:S_char} also suggests that if $N > d+1$ (even with just an oversampling of one point), the construction of such USBP operators becomes viable.
\section{Constructing USBP operators}
\label{sec:construction}

We now provide a simple procedure to construct degree $d$ USBP operators $D_{\pm} = P^{-1}( Q_{\pm} + B/2 )$.
While previous works, such as \cite{mattsson2017diagonal}, have developed diagonal-norm FD-USBP operators, our method encompasses a broader class of USBP operators. 
Importantly, it applies to arbitrary grid points, not limited to equidistant spacing.
This allows us to derive novel USBP operators on LGL points suited for nodal DG methods.

\subsection{Proposed procedure}
\label{sub:construction_procedure}

The idea behind our construction procedure is to start from a degree $d$ SBP operator $D = P^{-1}(Q + B/2)$ and a dissipation matrix $S$ with $S \mathbf{f} = \mathbf{0}$ if $f \in \mathcal{P}_d$ and $\mathbf{f}^T S \mathbf{f} < 0$ otherwise. 
Then part (a) of \cref{thm:existence} implies that $D_{\pm} = P^{-1}(Q_{\pm} + B/2)$ with $Q_{\pm} = Q \pm S/2$ are degree $d$ USBP operators. 
This observation motivates the following simple construction procedure: 
\begin{enumerate}

	\item[(P1)] 
	Construct a degree $d$ SBP operator $D = P^{-1}( Q + B/2 )$ on the grid $\mathbf{x} = [x_1,\dots,x_N]^T$ with $N > d+1$. 
	
	\item[(P2)] 
	Find a dissipation matrix $S$ with $S \mathbf{f} = \mathbf{0}$ if $f \in \mathcal{P}_d$ and $\mathbf{f}^T S \mathbf{f} < 0$ otherwise. 
	
	\item[(P3)] 
	Choose $Q_{\pm} = Q \, \pm \, S/2$ and get degree $d$ USBP operators $D_{\pm} = P^{-1}( Q_{\pm} + B/2 )$ with $Q_+ + Q_+^T = S$ on the grid $\mathbf{x}$. 
	
\end{enumerate} 
In (P3), equivalently, we get degree $d$ USBP operators as $D_{\pm} = D \pm P^{-1} S$. 
Note that we choose $N > d+1$ in (P2) because \cref{lem:S_char} states that we can find a dissipation matrix $S \neq 0$ if and only if $N > d+1$. 

The first step in our construction procedure, (P1), is to find a degree $d$ SBP operator on a grid $\mathbf{x} = [x_1,\dots,x_N]^T$ with $N > d+1$.
Several works have discussed the construction of SBP operators. 
See \cite{hicken2013summation,linders2018order, glaubitz2022summation,glaubitz2024optimization} or the reviews \cite{svard2014review,fernandez2014review} and references therein. 
Therefore, we omit such a discussion and focus instead on the construction of $S$ in step (P2), which is addressed in \Cref{sub:constrution_S}, along with a robust implementation strategy in \Cref{sub:DOPs}.
Furthermore, we exemplify the above procedure in \Cref{sub:example}.

\subsection{Constructing a dissipation matrix}
\label{sub:constrution_S}

The second step of our construction procedure, (P2), is to find a dissipation matrix $S$ with $S \mathbf{f} = \mathbf{0}$ if $f \in \mathcal{P}_d$ and $\mathbf{f}^T S \mathbf{f} < 0$ otherwise.
To this end, we can follow the proof of \cref{lem:S_char}.
Let us briefly recall the most important steps and comment on their implementation for completeness.
Since $S$ is real and symmetric, we can decompose it as
\begin{equation}\label{eq:constr_S_diag}
	S = V \Lambda V^{-1},
\end{equation}
where $\Lambda = \diag(\lambda_1,\dots,\lambda_N)$ contains the eigenvalues of $S$ and the columns of $V = [\mathbf{v_1},\dots,\mathbf{v_N}]$ are linearly independent eigenvectors of $S$.
(See \cite[Theorem 8.1.1]{golub2012matrix}.)
Let $\{ f_k \}_{k=1}^{d+1}$ be a basis of $\mathcal{P}_d$.
We can ensure $S \mathbf{f} = \mathbf{0}$ if $f \in \mathcal{P}_d$ (the dissipation matrix does not act on resolved modes) by using the nodal value vectors of the basis functions as the first $d+1$ eigenvectors with the corresponding eigenvalues being zero, i.e.,
\begin{equation}
	\mathbf{v_k} = \mathbf{f_k}, \quad \lambda_k = 0, \quad k=1,\dots,d+1.
\end{equation}
These are linearly independent since $\mathcal{P}_d$ is a Chebyshev system.
Now suppose that $N > d+1$.
Then, we still have $N-d-1 > 0$ free eigenvalues and eigenvectors.
These have to be chosen such that (i) the eigenvectors $\mathbf{v_{d+2}},\dots,\mathbf{v_N}$ are linearly independent among themselves and to the first $d+1$ eigenvectors and (ii) one has $\mathbf{f}^T S \mathbf{f} < 0$ if $f \not\in \mathcal{P}_d$.
We can ensure these properties by choosing the eigenvalues $\lambda_{d+2},\dots,\lambda_{N}$ as arbitrary \emph{negative} numbers, e.g., $\lambda_{d+2},\dots,\lambda_{N}=-1$, and $\mathbf{v_{d+2}},\dots,\mathbf{v_N}$ as the corresponding eigenvectors satisfying $S \mathbf{v_j} = \lambda_{j} \mathbf{v_j}$ for $j=d+2,\dots,N$.
Then, $S$ is a desired dissipation matrix, and we can formally find it by computing the matrix product on the right-hand side of \cref{eq:constr_S_diag}.
That said, computing the matrix product in \cref{eq:constr_S_diag} is generally not recommended since it involves matrix inversion.
We discuss an approach to bypass this problem in \Cref{sub:DOPs}.

\subsection{Robust implementation using discrete orthogonal polynomials}
\label{sub:DOPs}

We present a simple and robust method for computing the dissipation matrix $S$ in \cref{eq:constr_S_diag} using discrete orthogonal polynomials (DOPs), see \cite{gautschi2004orthogonal}. 
Suppose we are given a grid $\mathbf{x} = [x_1,\dots,x_N]^T$.
A basis $\{ f_k \}_{k=1}^{N}$ of $\mathcal{P}_N$ consists of \emph{DOPs} if the $f_k$'s are orthogonal with respect to the discrete inner product $[f,g]_{\mathbf{x}} = \sum_{n=1}^N f(x_n) g(x_n)$, i.e.,
\begin{equation}\label{eq:DOP}
	\sum_{n=1}^N f_j(x_n) f_k(x_n) = \delta_{j,k}, \quad
	j,k = 1,\dots,N,
\end{equation}
where $\| f_k \|^2_{\mathbf{x}} = [f_k,f_k]_{\mathbf{x}}$ is the induced norm.
Note that \cref{eq:DOP} is equivalent to $\scp{\mathbf{f_j}}{\mathbf{f_k}} = \delta_{j,k}$, where $\scp{\mathbf{f}}{\mathbf{g}} = \mathbf{f}^T \mathbf{g}$ and $\| \mathbf{f} \|^2 = \mathbf{f}^T \mathbf{f}$ are the usual Euclidean inner product and norm for $\mathbf{f}, \mathbf{g} \in \R^N$, respectively.
Hence, if we choose $\mathbf{v_k} = \mathbf{f_k}$ for the eigenvectors in \cref{eq:constr_S_diag}, where $\{f_k\}_{k=1}^N$ is a DOP basis of $\mathcal{P}_N$, then $V = [\mathbf{v_1},\dots,\mathbf{v_N}]$ is orthogonal ($V^{-1} = V^T$) and \cref{eq:constr_S_diag} becomes
\begin{equation}\label{eq:constr_S_diag2}
	S = V \Lambda V^T.
\end{equation}
The advantage of \cref{eq:constr_S_diag2} over \cref{eq:constr_S_diag} is that we no longer need to invert $V$.
To summarize the above discussion, we can efficiently and robustly construct a dissipation matrix $S$ as follows:
\begin{enumerate}

	\item[(S1)]
	Find a DOP basis $\{ f_k \}_{k=1}^N$ of $\mathcal{P}_N$ with respect to the grid $\mathbf{x} = [x_1,\dots,x_N]^T$.

	\item[(S2)]
	Choose the columns of $V$ as the vectors of nodal values of the DOP basis, i.e., $\mathbf{v_k} = \mathbf{f_k}$ for $k=1,\dots,N$.

	\item[(S3)]
	Set the first $d+1$ diagonal entries of $\Lambda$ in \cref{eq:constr_S_diag2} to zero, i.e., $\lambda_k = 0$ for $k=1,\dots,d+1$.\footnote{We assume that the first $d+1$ DOPs form a basis of $\mathcal{P}_d$. This ensures that the dissipation operator $S$ does not act on the resolved modes corresponding to polynomials of degree up to $d$.}

	\item[(S4)]
	Choose the remaining eigenvalues $\lambda_{d+2},\dots,\lambda_N < 0$ and compute $S$ using \cref{eq:constr_S_diag2}.

\end{enumerate}

A few remarks are in order. 

\begin{remark}
	Since the dissipation matrix $S$ is symmetric, the eigenvectors corresponding to \emph{different} eigenvalues are automatically orthogonal. 
	However, in our case, we encounter multiple linear independent eigenvectors associated with the same eigenvalue (for instance, $\lambda = 0$). 
	Consequently, these are not guaranteed to be orthogonal; therefore, we rely on the additional orthogonalization step described above. 
\end{remark}

\begin{remark}[Finding DOP bases]
	The existing literature offers explicit formulas of DOPs for different point distributions $\mathbf{x}$, including Chebyshev--Gauss--Lobatto \cite{eisinberg2007discrete} and equidistant points \cite{wilson1970hahn}.
	Suppose an explicit formula is not readily available. 
	In that case, we can start from a basis of Legendre polynomials, orthogonal with respect to the continuous $L^2$-inner product, and transform them into a DOP basis using the modified Gram--Schmidt process.
	See \cite{gautschi2004orthogonal,glaubitz2020stable} and references therein.
\end{remark}

\subsection{An example: USBP operators on LGL points}
\label{sub:example}

We exemplify the construction of USBP operators for a degree one ($d=1$) diagonal-norm USBP operator for three ($N=3$) LGL points on $\Omega_{\rm ref} = [-1,1]$.
This case allows us to represent exactly all involved matrices below.
Additional diagonal-norm USBP operators on up to six LGL nodes are provided in \cref{app:operators}. 
Moreover, \cref{app:examples} exemplifies the construction of USBP operators for Gauss--Legendre nodes (that do not include the boundary points) and a dense-norm USBP operator on equidistant points. 

Let us first address (P1):
To compute a degree one diagonal-norm SBP operator on three LGL points, we have two options:
The first option is to determine a diagonal-norm SBP operator with exactly degree one, i.e., the operator is exact for linear functions but not necessarily for quadratic ones.
We can find such an operator following the procedure described in \cite{hicken2013summation,linders2018order,glaubitz2022summation}.
Note that this operator is not unique but comes with one degree of freedom that can be used to optimize the SBP operator based on different criteria, such as the truncation error and spectral properties \cite{strand1994summation,svard2014review,fernandez2014review}.
The second option is to determine a degree two diagonal-norm SBP operator.
This operator is unique and readily available by evaluating the derivatives of the corresponding Lagrange basis functions at the LGL grid points \cite{gassner2013skew,carpenter2014entropy,carpenter2015entropy,chen2017entropy,cicchino2022nonlinearly,montoya2022unifying}.
We follow the second option since it is computationally more convenient.
The associated grid points, quadrature weights (diagonal entries of $P$), and SBP operator on $\Omega_{\rm ref} = [-1,1]$ are
\begin{equation}\label{eq:USBP_ex_D}
	\mathbf{x} =
	\begin{bmatrix}
		-1 \\ 0 \\ 1
	\end{bmatrix}, \quad
	\mathbf{p} = \frac{1}{3}
	\begin{bmatrix}
		1 \\ 4 \\ 1
	\end{bmatrix}, \quad
	D = \frac{1}{2}
	\begin{bmatrix}
		-3 & 4 & -1 \\
		-1 & 0 & 1 \\
		1 & -4 & 3
	\end{bmatrix}.
\end{equation}
The boundary matrix is $B = \diag(-1,0,1)$ since the boundary points are included in the grid.

We next address (P2), finding a dissipation matrix $S$ with $S \mathbf{f} = \mathbf{0}$ if $f \in \mathcal{P}_1$ and $\mathbf{f}^T S \mathbf{f} < 0$ otherwise.
To this end, we follow \Cref{sub:DOPs} and robustly construct $S$ using a DOP basis on the three LGL grid points.
We get these by starting from the first three Legendre polynomials, spanning $\mathcal{P}_3$, and then applying the Gram--Schmidt procedure.
The Vandermonde matrix of this DOP basis is
\begin{equation}
	V = \frac{1}{\sqrt{6}}
	\begin{bmatrix}
		\sqrt{2} & -\sqrt{3} & 1 \\
		\sqrt{2} & 0 & -2 \\
		\sqrt{2} & \sqrt{3} & 1
	\end{bmatrix}.
\end{equation}
Consequently, we get the dissipation matrix as $S = V \Lambda V^T$ with $\Lambda = \diag(\lambda_1,\lambda_2,\lambda_3)$.
To ensure that the desired USBP operator has degree one ($d=1$), we choose the first two eigenvalues as zero, i.e., $\lambda_1=\lambda_2=0$.
Furthermore, we choose $\lambda_3 = -1$, which introduces artificial dissipation to the highest, unresolved, mode.\footnote{
Note that $\lambda_3 = -1$ is an arbitrary choice.
Any negative value yields an admissible dissipation matrix.
We investigate the influence of the choice of $\lambda_3$ in more detail in \Cref{sec:tests}.
}
The resulting dissipation matrix is
\begin{equation}
	S = \frac{1}{6}
	\begin{bmatrix}
		-1 & 2 & -1 \\
		2 & -4 & 2 \\
		-1 & 2 & -1
	\end{bmatrix}.
\end{equation}

Finally, (P3), we get degree one diagonal-norm USBP operators as $D_{\pm} = D \pm P^{-1}S/2$, yielding
\begin{equation}
	D_+ = \frac{1}{8}
	\begin{bmatrix}
		-14 & 20 & -6 \\
		-3 & -2 & 5 \\
		2 & -12 & 10
	\end{bmatrix}, \quad
	D_- = \frac{1}{8}
	\begin{bmatrix}
		-10 & 12 & -2 \\
		-5 & 2 & 3 \\
		6 & -20 & 14
	\end{bmatrix}.
\end{equation}
The norm and boundary matrix of the degree one USBP operators are the same as those of the degree two SBP operator.
\section{Application of USBP operators to nodal DG methods} 
\label{sec:application} 

We are now positioned to combine the above-developed DG-USBP operators with flux vector splitting approaches to construct new DG-USBP baseline schemes for nonlinear conservation laws.

\subsection{The flux vector splitting approach}
\label{sub:splitting}

We briefly review the classical flux vector splitting approach. 
See \cite{steger1979flux,vanleer1982flux,buning1982solution,hanel1987accuracy,liou1991high,coirier1991numerical} as well as \cite[Chapter 8]{toro2013riemann} and references therein for more details. 
Consider the generic one-dimensional hyperbolic conservation law 
\begin{equation}\label{eq:consLaw} 
	\partial_t u + \partial_x f( u ) = 0, \quad x \in \Omega = (x_{\rm min},x_{\rm max}),
\end{equation} 
with conserved variable $u = u(t,x)$ and flux $f = f(u)$.  
Henceforth, we assume that \cref{eq:consLaw} is equipped with appropriate initial and boundary conditions. 
In flux vector splitting approaches, the flux $f$ is split into two parts, 
\begin{equation}\label{eq:flux_vector_splitting}
	f(u) = f_-(u) + f_+(u),
\end{equation}
where the eigenvalues $\lambda^{\pm}_i(u)$ of the Jacobian $\nabla_u f_{\pm}(u)$ satisfy $\lambda_i^-(u) \leq 0$ and $\lambda_i^+(u) \geq 0$ for all $i$, respectively. 
Applying \cref{eq:flux_vector_splitting} to \cref{eq:consLaw}, we get 
\begin{equation}\label{eq:consLaw_splitt} 
	\partial_t u + \partial_x f_-( u ) + \partial_x f_+( u ) = 0, \quad x \in \Omega = [x_{\rm min},x_{\rm max}].
\end{equation}
We can now semi-discretize \cref{eq:consLaw_splitt} using `global' USBP operators $D_{\pm}$ on $\Omega = [x_{\rm min},x_{\rm max}]$, resulting in 
\begin{equation}\label{eq:consLaw_splitt_discr}
	\frac{\d}{\d t} \mathbf{u} + D_+ \mathbf{f_-} + D_- \mathbf{f_+} = 0,
\end{equation}
where $\mathbf{u}$ and $\mathbf{f_{\pm}}$ are the nodal values of the numerical solution and the two flux components at the global grid points.
Note that we have ignored boundary conditions for the moment. 

\begin{remark}
  	The mismatch between the indices $\pm$ of the upwind operators and the fluxes is intentional. 
  	This choice is made to maintain backward compatibility, preserving the established conventions of both flux vector splitting methods and upwind SBP operators for historical consistency. 
\end{remark}

\subsection{Application to nodal collocated DG methods}
\label{sub:DG}

In DG-type methods, the computational domain $\Omega = [x_{\rm min},x_{\rm max}]$ is partitioned into $J$ disjoint elements $\Omega_j$ with $\bigcup_{j=1}^J \Omega_j = \Omega$. 
Each element $\Omega_j$ is mapped diffeomorphically onto a reference element, say $\Omega_{\rm ref} = [-1,1]$, in which all computations are carried out. 
This has the advantage of only computing USBP operators $D_{\pm}$ once for the reference element rather than for each element separately. 
Applying this DG-type approach to \cref{eq:consLaw_splitt}, we get the DG-USBP semi-discretization 
\begin{equation}\label{eq:SE_discr1}
	\frac{\d}{\d t} \mathbf{u}^j + D_+ \mathbf{f}_-^j + D_- \mathbf{f}_+^j = \mathbf{0}, \quad j=1,\dots,J,
\end{equation}
where $\mathbf{u}^j$ and $\mathbf{f}_{\pm}^j$ are the nodal values of the numerical solution and the two flux vector splitting components on the $j$th element.
We allow for the numerical solution to be discontinuous at element interfaces. 
Hence, at each element interface, we have two values. 
We distinguish between values from inside $\Omega_j$, e.g., $u^j_L$ and $u^j_R$, and values from the neighboring elements, $\Omega_{j-1}$ and $\Omega_{j+1}$, e.g., $u^{j-1}_R$ and $u^{j+1}_L$. 
Thereby, the subindices ``L" and ``R" denote that the value of the numerical solution at the left and right element boundary is taken, respectively. 
From conservation and stability (upwinding) considerations---and to allow for information to be shared between neighboring elements---we introduce a weak coupling between elements in \cref{eq:SE_discr1} using simultaneous approximation terms (SATs), resulting in the semi-discretization 
\begin{equation}\label{eq:USBP_SAT}
	\frac{\d}{\d t} \mathbf{u}^j + D_+ \mathbf{f}_-^j + D_- \mathbf{f}_+^j = \mathbf{SAT}^j, \quad j=1,\dots,J.
\end{equation} 
Please see \cite{svard2014review,fernandez2014review,carpenter2014entropy,del2018simultaneous,chen2020review,manzanero2020entropy2} and references therein for details on SATs. 
Here, the SAT in the $j$-th element is 
\begin{equation}\label{eq:SAT_numFlux}
	\mathbf{SAT}^j = 
		- P^{-1} \mathbf{e}_R \left[ \fnum( u_R^{j}, u_L^{j+1} ) - f^j_R \right] 
		+ P^{-1} \mathbf{e}_L \left[ \fnum( u_R^{j-1}, u_L^{j} ) - f_L^j \right], 
\end{equation} 
where $u_{L/R}^j = \mathbf{e}_{L/R}^T \mathbf{u}^j$ and $f_{L/R}^j = \mathbf{e}_{L/R}^T \mathbf{f}^j$ abbreviate the values of the numerical solution and the corresponding values of the flux at the left/right boundary of the element $\Omega_j$. 
For simplicity, we suppose that the grid includes the boundary points, i.e., $x_1 = -1$ and $x_N = 1$ in the reference element. 
Otherwise, one has to incorporate certain projection operators, which introduce additional technical details that we omit here.
If the boundary points are included, $\mathbf{e}_L = [1,0,\dots,0]^T$ and $\mathbf{e}_R = [0,\dots,0,1]^T$, $u_{L/R}^{j} = u_{1/N}^j$, and $f_{L/R}^{j} = f_{1/N}^j$. 
Furthermore, $\fnum( u_R^{j}, u_L^{j+1} )$ and $\fnum( u_R^{j-1}, u_L^{j} )$ are two-point numerical flux functions at the right and left boundary of the element $\Omega_j$, respectively. 
There are many existing numerical fluxes that we can use. 
Alternatively, as described in \cite{ranocha2023high}, we can use the flux vector splitting to design numerical fluxes. 
In this approach, we use the same splitting for $f$ and $\fnum$, yielding 
\begin{equation}\label{eq:spitting_numFlux}
\begin{aligned} 
	\fnum( u_R^{j}, u_L^{j+1} ) - f^j_R 
		& = \left[ f_+( u_R^{j} ) + f_-( u_L^{j+1} ) \right] 
			- \left[ f_+( u_R^{j} ) + f_-( u_R^{j} ) \right] 
		= f_-( u_L^{j+1} ) - f_-( u_R^{j} ), \\ 
	\fnum( u_R^{j-1}, u_L^{j} ) - f_L^j 
		& = \left[ f_+( u_R^{j-1} ) + f_-( u_L^{j} ) \right] 
			- \left[ f_+( u_L^{j} ) + f_-( u_L^{j} ) \right] 
		= f_+( u_R^{j-1} ) - f_+( u_L^{j} ). 
\end{aligned}	
\end{equation} 
Substituting \cref{eq:spitting_numFlux} into the USBP-SAT semi-discretization \cref{eq:USBP_SAT} results in 
\begin{equation}\label{eq:USBP_SAT2}
	\frac{\d}{\d t} \mathbf{u}^j + D_+ \mathbf{f}_-^j + D_- \mathbf{f}_+^j
		= - P^{-1} \mathbf{e}_R \left[ f_-( u_L^{j+1} ) - f_-( u_R^{j} ) \right] 
		+ P^{-1} \mathbf{e}_L \left[ f_+( u_R^{j-1} ) - f_+( u_L^{j} ) \right]
\end{equation} 
for $j=1,\dots,N$. 
In our implementation, we opted for the scheme given by \cref{eq:USBP_SAT2}. 
This choice has the advantage of preventing the need to resolve a Riemann problem in terms of numerical fluxes, which streamlines the implementation and improves efficiency. 
Finally, one evolves the numerical solution in time by solving the semi-discrete equation \cref{eq:USBP_SAT2} with an appropriate ODE solver. 
\section{Numerical tests}
\label{sec:tests}

We now examine the robustness properties of the DG-USBP method described in \Cref{sec:application} based on the new DG-USBP operators we constructed in \Cref{sec:construction}. 
Specifically, we focus on diagonal-norm USBP operators on LGL nodes, which are predominantly used in spectral element methods since they enable splitting techniques \cite{nordstrom2006conservative,gassner2016split,offner2019error} and the extension to variable coefficients including curvilinear coordinates \cite{chan2019efficient,svard2004coordinate,nordstrom2017conservation,parsani2016entropy}. 
In many cases, we compare the DG-USBP method with an entropy-stable flux differencing DGSEM on the same LGL points.
For brevity, we will henceforth refer to the entropy-stable flux-differencing DGSEM simply as ``DGSEM.'' 
The numerical experiments are conducted using the Julia programming language \cite{bezanson2017julia}.
We employ Runge-Kutta methods from the OrdinaryDiffEq.jl package for time integration  \cite{rackauckas2017differentialequations}.
Spatial discretization is handled using the Trixi.jl framework \cite{ranocha2022adaptive} in combination with the SummationByPartsOperators.jl package \cite{ranocha2021sbp}.
Visualization of the results is achieved through Plots.jl \cite{christ2023plots} and ParaView \cite{ahrens2005paraview}.
All source code required to reproduce the numerical experiments is available online in the reproducibility repository \cite{glaubitz2024UpwindRepro}.

\subsection{Convergence for the linear advection equation}
\label{sub:convergence_advection}

We initiate our study with a convergence analysis of the proposed USBP operators for the one-dimensional scalar linear advection equation
\begin{equation}\label{eq:linear_adv_tests}
\begin{aligned}
  	\partial_t u(t, x) + \partial_x u(t, x)
  		& = 0,
		&& t \in (0, 5), \ x \in (-1, 1),
		\\
  	u(0, x)
		& = \sin(\pi x),
		&& x \in [-1, 1],
\end{aligned}
\end{equation}
equipped with periodic boundary conditions.
The analysis utilizes the classical Lax--Friedrichs flux vector splitting with $\lambda_{\rm max} = 1$, i.e., $f^-(u) = 0$ and $f^+(u) = u$.
For time integration, we employ the nine-stage, fourth-order accurate Runge--Kutta (RK) method of \cite{ranocha2021optimized} with error-based step size control and a tolerance of $10^{-12}$ to integrate the semi-discretizations in time.
We measure the discrete $L^2$ error using the quadrature rule associated with the norm matrix $P$.

\begin{table}[tb]
	\centering
	\begin{adjustbox}{width=0.95\textwidth}
  	\begin{tabular}{rrr}
  		\multicolumn{3}{c}{$\lambda_3 = -10^{-3}$} \\
		\toprule
  		$J$ & $L^2$ error & EOC \\
		\midrule
    		2 & 7.41e-01 & \\
    		4 & 1.22e-01 & 2.61 \\
    		8 & 1.06e-02 & 3.53 \\
   		16 & 9.75e-04 & 3.44 \\
   		32 & 1.07e-04 & 3.18 \\
   		64 & 1.29e-05 & 3.06 \\
  		128 & 1.60e-06 & 3.01 \\
		\bottomrule
	\end{tabular}
	\hspace{.05cm}
  	\begin{tabular}{rrr}
  		\multicolumn{3}{c}{$\lambda_3 = -10^{-2}$} \\
		\toprule
  		$J$ & $L^2$ error & EOC \\
		\midrule
    		2 & 7.38e-01 & \\
    		4 & 1.21e-01 & 2.61 \\
    		8 & 1.05e-02 & 3.52 \\
   		16 & 9.74e-04 & 3.43 \\
   		32 & 1.08e-04 & 3.17 \\
   		64 & 1.34e-05 & 3.01 \\
  		128 & 1.86e-06 & 2.85 \\
		\bottomrule
	\end{tabular}
	\hspace{.05cm}
  	\begin{tabular}{rrr}
  		\multicolumn{3}{c}{$\lambda_3 = -10^{-1}$} \\
		\toprule
  		$J$ & $L^2$ error & EOC \\
		\midrule
    		2 & 7.47e-01 & \\
    		4 & 2.09e-01 & 1.83 \\
    		8 & 3.37e-02 & 2.63 \\
   		16 & 5.32e-03 & 2.66 \\
   		32 & 1.02e-03 & 2.38 \\
   		64 & 2.31e-04 & 2.14 \\
  		128 & 5.61e-05 & 2.04 \\
		\bottomrule
	\end{tabular}
	\hspace{.05cm}
  	\begin{tabular}{rrr}
  		\multicolumn{3}{c}{$\lambda_3 = -1$} \\
		\toprule
  		$J$ & $L^2$ error & EOC \\
		\midrule
    		2 & 8.14e-01 & \\
    		4 & 3.52e-01 & 1.21 \\
    		8 & 6.61e-02 & 2.41 \\
   		16 & 1.08e-02 & 2.61 \\
   		32 & 2.09e-03 & 2.37 \\
   		64 & 4.75e-04 & 2.14 \\
  		128 & 1.15e-04 & 2.04 \\
		\bottomrule
	\end{tabular}
	\end{adjustbox}
	\\
	\vspace{.3cm}
	\begin{adjustbox}{width=0.95\textwidth}
  	\begin{tabular}{rrr}
  		\multicolumn{3}{c}{$\lambda_4 = -10^{-3}$} \\
		\toprule
  		$J$ & $L^2$ error & EOC \\
		\midrule
    		2 & 1.06e-01 & \\
    		4 & 4.66e-03 & 4.51 \\
    		8 & 2.73e-04 & 4.09 \\
   		16 & 1.70e-05 & 4.00 \\
   		32 & 1.07e-06 & 4.00 \\
  	 	64 & 6.66e-08 & 4.00 \\
  		128 & 4.18e-09 & 4.00 \\
		\bottomrule
	\end{tabular}
	\hspace{.05cm}
  	\begin{tabular}{rrr}
  		\multicolumn{3}{c}{$\lambda_4 = -10^{-2}$} \\
		\toprule
  		$J$ & $L^2$ error & EOC \\
		\midrule
    		2 & 1.06e-01 &  \\
    		4 & 4.64e-03 & 4.51 \\
    		8 & 2.72e-04 & 4.09 \\
   		16 & 1.70e-05 & 4.00 \\
   		32 & 1.08e-06 & 3.98 \\
   		64 & 7.16e-08 & 3.92 \\
  		128 & 5.34e-09 & 3.74 \\
		\bottomrule
	\end{tabular}
	\hspace{.05cm}
  	\begin{tabular}{rrr}
  		\multicolumn{3}{c}{$\lambda_4 = -10^{-1}$} \\
		\toprule
  		$J$ & $L^2$ error & EOC \\
		\midrule
    		2 & 1.02e-01 & \\
    		4 & 4.63e-03 & 4.46 \\
    		8 & 2.92e-04 & 3.99 \\
   		16 & 2.30e-05 & 3.67 \\
   		32 & 2.27e-06 & 3.34 \\
   		64 & 2.62e-07 & 3.12 \\
  		128 & 3.21e-08 & 3.03 \\
		\bottomrule
	\end{tabular}
	\hspace{.05cm}
  	\begin{tabular}{rrr}
  		\multicolumn{3}{c}{$\lambda_4 = -1$} \\
		\toprule
  		$J$ & $L^2$ error & EOC \\
		\midrule
    		2 & 1.02e-01 & \\
    		4 & 4.63e-03 & 4.46 \\
    		8 & 2.92e-04 & 3.99 \\
   		16 & 2.30e-05 & 3.67 \\
   		32 & 2.27e-06 & 3.34 \\
   		64 & 2.62e-07 & 3.12 \\
  		128 & 3.21e-08 & 3.03 \\
		\bottomrule
	\end{tabular}
	\end{adjustbox}
	\\
	\vspace{.3cm}
	\begin{adjustbox}{width=0.95\textwidth}
  	\begin{tabular}{rrr}
  		\multicolumn{3}{c}{$\lambda_5 = -10^{-3}$} \\
		\toprule
  		$J$ & $L^2$ error & EOC \\
		\midrule
    		2 & 6.01e-03 & \\
    		4 & 3.04e-04 & 4.31 \\
    		8 & 9.75e-06 & 4.96 \\
   		16 & 3.07e-07 & 4.99 \\
   		32 & 9.62e-09 & 5.00 \\
   		64 & 3.01e-10 & 5.00 \\
  		128 & 9.47e-12 & 4.99 \\
		\bottomrule
	\end{tabular}
	\hspace{.05cm}
  	\begin{tabular}{rrr}
  		\multicolumn{3}{c}{$\lambda_5 = -10^{-2}$} \\
		\toprule
  		$J$ & $L^2$ error & EOC \\
		\midrule
    		2 & 6.11e-03 & \\
    		4 & 3.03e-04 & 4.34 \\
    		8 & 9.72e-06 & 4.96 \\
   		16 & 3.08e-07 & 4.98 \\
   		32 & 9.88e-09 & 4.96 \\
   		64 & 3.36e-10 & 4.88 \\
  		128 & 1.34e-11 & 4.65 \\
		\bottomrule
	\end{tabular}
	\hspace{.05cm}
  	\begin{tabular}{rrr}
  		\multicolumn{3}{c}{$\lambda_5 = -10^{-1}$} \\
		\toprule
  		$J$ & $L^2$ error & EOC \\
		\midrule
    		2 & 7.17e-03 & \\
    		4 & 3.05e-04 & 4.56 \\
    		8 & 1.10e-05 & 4.79 \\
   		16 & 4.71e-07 & 4.54 \\
   		32 & 2.48e-08 & 4.25 \\
   		64 & 1.47e-09 & 4.08 \\
  		128 & 9.05e-11 & 4.02 \\
		\bottomrule
	\end{tabular}
	\hspace{.05cm}
  	\begin{tabular}{rrr}
  		\multicolumn{3}{c}{$\lambda_5 = -1$} \\
		\toprule
  		$J$ & $L^2$ error & EOC \\
		\midrule
    		2 & 1.62e-02 & \\
    		4 & 6.42e-04 & 4.66 \\
    		8 & 3.89e-05 & 4.05 \\
   		16 & 2.42e-06 & 4.01 \\
   		32 & 1.51e-07 & 4.00 \\
   		64 & 9.43e-09 & 4.00 \\
  		128 & 5.89e-10 & 4.00 \\
		\bottomrule
	\end{tabular}
	\end{adjustbox}
	\caption{
		Convergence results using DG-USBP discretizations of the linear advection equation with Lax--Friedrichs splitting and $J$ elements.
		We use degree one (first row), two (second row), and three (third row) USBP operators on three, fours, and five LGL nodes per element, respectively, with different values for the artificial dissipation parameters $\lambda_3$, $\lambda_4$, and $\lambda_5$.
	}
  	\label{tab:convergence_linear}
\end{table}

\cref{tab:convergence_linear} presents the findings of our convergence study, detailing the discrete $L^2$ error using the quadrature rule associated with the norm matrix $P$ and the resulting experimental orders of convergence (EOC) for degree one ($d=1$), degree two ($d=2$), and three ($d=3$) USBP operators on three ($N=3$), four ($N=4$), and five ($N=5$) LGL nodes, respectively.
These USBP operators are characterized by a non-trivial dissipation matrix $S = V \Lambda V^T$, where $V$ denotes an orthogonal Vandermonde matrix, and $\Lambda = \diag(\lambda_1,\dots,\lambda_N)$ encapsulates the dissipation parameters, which regulate the dissipation added to unresolved modes.
In the instance of the degree one ($d=1$) USBP operator on three ($N=3$) LGL nodes, the parameters $\lambda_1$ and $\lambda_2$ are set to zero, ensuring that the operator is exact for polynomials up to degree one.
Conversely, $\lambda_3 \leq 0$ introduces artificial dissipation to unresolved modes.
The impact of different $\lambda_3$ values on the accuracy and convergence behavior of the USBP scheme is delineated in the top row of \cref{tab:convergence_linear}.
Setting $\lambda_3$ to zero recovers the central degree three SBP operator on three LGL points, yielding fourth-order convergence.
A gradual decrease in $\lambda_3$ results in a reduction of the EOC from three to two.
Moreover, as we will show, lower values of $\lambda_3$ diminish the scheme's accuracy but enhance its robustness.
Regarding the degree two ($d=2$) USBP operator on four ($N=4$) LGL nodes, the parameters $\lambda_1$ through $\lambda_3$ are zeroed, rendering the operator exact for polynomials up to degree two.
This time, the parameter $\lambda_4 \leq 0$ adds artificial dissipation to unresolved modes.
The second row of \cref{tab:convergence_linear} examines the effects of varying $\lambda_4$ on the USBP scheme's accuracy and convergence.
Decreasing $\lambda_4$ reduces the EOC from four to three due to adding artificial dissipation to the scheme. 
For the degree three ($d=3$) USBP operator on five ($N=5$) LGL nodes, the parameters $\lambda_1$ through $\lambda_4$ are zeroed, rendering the operator exact for polynomials up to degree three.
The third row of \cref{tab:convergence_linear} reports the effects of varying $\lambda_5$ on the USBP scheme's accuracy and convergence.
Decreasing $\lambda_5$ reduces the EOC from five to four.

\subsubsection*{Discussion of the results}

The above findings indicate that combining flux vector splitting techniques with DG-USBP operators does not lead to excessive artificial dissipation, similar to high-order FD-USBP schemes.

\subsection{Convergence for the compressible Euler equations}
\label{sub:convergence_Euler}

We continue our study with a convergence analysis for the one-dimensional compressible Euler equations
\begin{equation}\label{eq:euler}
  \partial_t \begin{pmatrix} \rho \\ \rho v \\ \rho e \end{pmatrix}
  + \partial_x \begin{pmatrix} \rho v \\ \rho v^2 + p \\ (\rho e + p) v \end{pmatrix}
  = 0
\end{equation}
of an ideal gas with density $\rho$, velocity $v$, total energy density $\rho e$, and pressure $p = (\gamma - 1) \left( \rho e - \rho v^2 / 2 \right)$, where the ratio of specific heat is chosen as $\gamma = 1.4$.
We add a source term to manufacture the solution
\begin{equation}
  \rho(t, x) = h(t, x), \quad v(t, x) = 1, \quad \rho e(t, x) = h(t, x)^2,
\end{equation}
with $h(t, x) = 2 + 0.1 \sin\bigl( \pi (x - t) \bigr)$ for $t \in [0, 2]$ and $x \in [0, 2]$.
The time-integration is done as in \Cref{sub:convergence_advection, here with a tolerance of $10^{-13}$.}

\begin{table}[tb]
	\centering
	\begin{adjustbox}{width=0.95\textwidth}
	\begin{tabular}{rrr}
  		\multicolumn{3}{c}{DGSEM, Ranocha flux} \\
		\toprule
  		$J$ & $L^2$ error & EOC \\
		\midrule
    		2 & 2.58e-02 & \\
    		4 & \textbf{4.00e-03} & 2.69 \\
    		8 & \textbf{3.75e-04} & 3.42 \\
   		16 & 6.05e-05 & 2.63 \\
   		32 & 8.62e-06 & 2.81 \\
   		64 & 1.14e-06 & 2.92 \\
  		128 & 1.43e-07 & 3.00 \\
		\bottomrule
	\end{tabular}
	\hspace{.05cm}
	\begin{tabular}{rrr}
  		\multicolumn{3}{c}{DGSEM, Shima flux} \\
		\toprule
  		$J$ & $L^2$ error & EOC \\
		\midrule
    		2 & \textbf{2.51e-02} & \\
    		4 & 4.06e-03 & 2.63 \\
    		8 & 3.98e-04 & 3.35 \\
   		16 & \textbf{5.42e-05} & 2.88 \\
   		32 & \textbf{7.13e-06} & 2.93 \\
   		64 & \textbf{8.47e-07} & 3.07 \\
  		128 & \textbf{1.03e-07} & 3.04 \\
		\bottomrule
	\end{tabular}
	\hspace{.05cm}
  	\begin{tabular}{rrr}
  		\multicolumn{3}{c}{$\lambda_3 = -10^{-3}$} \\
		\toprule
  		$J$ & $L^2$ error & EOC \\
		\midrule
    		2 & 3.74e-02 & \\
    		4 & 5.89e-03 & 2.66 \\
    		8 & 6.38e-04 & 3.21 \\
   		16 & 6.90e-05 & 3.21 \\
   		32 & 8.24e-06 & 3.07 \\
   		64 & 1.07e-06 & 2.95 \\
  		128 & 1.51e-07 & 2.82 \\
		\bottomrule
	\end{tabular}
	\hspace{.05cm}
  	\begin{tabular}{rrr}
  		\multicolumn{3}{c}{$\lambda_3 = -1$} \\
		\toprule
  		$J$ & $L^2$ error & EOC \\
		\midrule
    		2 & 3.38e-02 & \\
    		4 & 9.10e-03 & 1.89 \\
    		8 & 1.64e-03 & 2.48 \\
   		16 & 3.31e-04 & 2.31 \\
   		32 & 7.68e-05 & 2.11 \\
   		64 & 1.88e-05 & 2.03 \\
  		128 & 4.68e-06 & 2.01 \\
		\bottomrule
	\end{tabular}
	\end{adjustbox}
	\\
	\vspace{.3cm}
	\begin{adjustbox}{width=0.95\textwidth}
	\begin{tabular}{rrr}
  		\multicolumn{3}{c}{DGSEM, Ranocha flux} \\
		\toprule
  		$J$ & $L^2$ error & EOC \\
		\midrule
    		2 & \textbf{3.71e-03} & \\
    		4 & \textbf{2.06e-04} & 4.17 \\
    		8 & 2.79e-05 & 2.89 \\
   		16 & 3.32e-06 & 3.07 \\
   		32 & 2.67e-07 & 3.64 \\
   		64 & 1.36e-08 & 4.29 \\
  		128 & 7.51e-10 & 4.18 \\
		\bottomrule
	\end{tabular}
	\hspace{.05cm}
	\begin{tabular}{rrr}
  		\multicolumn{3}{c}{DGSEM, Shima flux} \\
		\toprule
  		$J$ & $L^2$ error & EOC \\
		\midrule
    		2 & 3.83e-03 & \\
    		4 & 2.47e-04 & 3.95 \\
    		8 & \textbf{1.98e-05} & 3.64 \\
   		16 & \textbf{1.75e-06} & 3.50 \\
   		32 & \textbf{8.62e-08} & 4.34 \\
   		64 & \textbf{4.56e-09} & 4.24 \\
  		128 & \textbf{2.75e-10} & 4.05 \\
		\bottomrule
	\end{tabular}
	\hspace{.05cm}
  	\begin{tabular}{rrr}
  		\multicolumn{3}{c}{$\lambda_4 = -10^{-3}$} \\
		\toprule
  		$J$ & $L^2$ error & EOC \\
		\midrule
    		2 & 7.13e-03 & \\
    		4 & 3.53e-04 & 4.33 \\
    		8 & 3.33e-05 & 3.41 \\
   		16 & 4.49e-06 & 2.89 \\
   		32 & 2.12e-07 & 4.41 \\
   		64 & 9.04e-09 & 4.55 \\
  		128 & 5.17e-10 & 4.13 \\
		\bottomrule
	\end{tabular}
	\hspace{.05cm}
  	\begin{tabular}{rrr}
  		\multicolumn{3}{c}{$\lambda_4 = -1$} \\
		\toprule
  		$J$ & $L^2$ error & EOC \\
		\midrule
    		2 & 4.88e-03 & \\
    		4 & 5.94e-04 & 3.04 \\
    		8 & 6.81e-05 & 3.12 \\
   		16 & 8.75e-06 & 2.96 \\
   		32 & 1.17e-06 & 2.91 \\
   		64 & 1.48e-07 & 2.98 \\
  		128 & 1.85e-08 & 3.00 \\
		\bottomrule
	\end{tabular}
	\end{adjustbox}
	\\
	\vspace{.3cm}
	\begin{adjustbox}{width=0.95\textwidth}
	\begin{tabular}{rrr}
  		\multicolumn{3}{c}{DGSEM, Ranocha flux} \\
		\toprule
  		$J$ & $L^2$ error & EOC \\
		\midrule
    		2 & \textbf{3.32e-04} & \\
    		4 & 1.87e-05 & 4.15 \\
    		8 & 1.73e-06 & 3.43 \\
   		16 & 1.46e-07 & 3.57 \\
   		32 & 6.16e-09 & 4.57 \\
   		64 & 1.25e-10 & 5.62 \\
  		128 & 3.51e-12 & 5.16 \\
		\bottomrule
	\end{tabular}
	\hspace{.05cm}
	\begin{tabular}{rrr}
  		\multicolumn{3}{c}{DGSEM, Shima flux} \\
		\toprule
  		$J$ & $L^2$ error & EOC \\
		\midrule
    		2 & 6.26e-04 & \\
    		4 & \textbf{1.80e-05} & 5.12 \\
    		8 & \textbf{7.43e-07} & 4.60 \\
   		16 & \textbf{3.61e-08} & 4.36 \\
   		32 & \textbf{8.69e-10} & 5.38 \\
   		64 & \textbf{2.12e-11} & 5.36 \\
  		128 & \textbf{9.15e-13} & 4.53 \\
		\bottomrule
	\end{tabular}
	\hspace{.05cm}
  	\begin{tabular}{rrr}
  		\multicolumn{3}{c}{$\lambda_5 = -10^{-3}$} \\
		\toprule
  		$J$ & $L^2$ error & EOC \\
		\midrule
    		2 & 7.79e-04 & \\
    		4 & 2.98e-05 & 4.71 \\
    		8 & 1.26e-06 & 4.57 \\
   		16 & 5.31e-08 & 4.56 \\
   		32 & 9.82e-10 & 5.76 \\
   		64 & 3.34e-11 & 4.88 \\
  		128 & 1.41e-12 & 4.57 \\
		\bottomrule
	\end{tabular}
	\hspace{.05cm}
  	\begin{tabular}{rrr}
  		\multicolumn{3}{c}{$\lambda_5 = -1$} \\
		\toprule
  		$J$ & $L^2$ error & EOC \\
		\midrule
    		2 & 1.21e-03 & \\
    		4 & 6.48e-05 & 4.22 \\
    		8 & 4.81e-06 & 3.75 \\
   		16 & 4.15e-07 & 3.53 \\
   		32 & 1.70e-08 & 4.61 \\
   		64 & 9.50e-10 & 4.16 \\
  		128 & 5.73e-11 & 4.05 \\
		\bottomrule
	\end{tabular}
	\end{adjustbox}
	\caption{
		Convergence results using the DGSEM and DG-USBP discretizations of the Euler equations for different numbers of elements $J$.
		For the DGSEM, we use polynomial degrees two (first row), three (second row), and four (third row) with the HLL surface flux, and the volume fluxes of Shima et al.\ and Ranocha.
		For the DG-USBP method, we use degree one (first row), two (second row), and three (third row) USBP operators with the van Leer--H\"anel splitting and different values for the artificial dissipation parameters $\lambda_3$, $\lambda_4$, and $\lambda_5$.
		For both methods, we use three (first row), four (second row), and five (third row) LGL nodes per element, respectively.
	}
  	\label{tab:convergence_Euler_vLH}
\end{table}

\cref{tab:convergence_Euler_vLH} summarizes the results of our convergence study for the van Leer--H\"anel flux vector splitting \cite{vanleer1982flux,hanel1987accuracy,liou1991high}.
We obtained similar results for the Steger--Warming flux vector splitting \cite{steger1979flux}, which are not reported here but can be reproduced using the openly available code repository \cite{glaubitz2024UpwindRepro}. 
We use degree one (first row), two (second row), and three (third row) USBP operators on three, four, and five LGL nodes per element, respectively, with different values for the artificial dissipation parameters $\lambda_3$, $\lambda_4$, and $\lambda_5$.
Similar to the convergence test for the linear advection equation in \Cref{sub:convergence_advection}, we observe that the EOC reduces by one as the respective dissipation parameter $\lambda_N$ is gradually decreased from $-10^{-3}$ to $-1$ for each operator.
Furthermore, we compare the performance of the DG-USBP method with a DGSEM employing the HLL surface flux \cite{harten1983upstream} in combination with the volume fluxes of Shima et al.\ \cite{shima2021preventing} and Ranocha \cite{ranocha2018comparison}.
To allow for a fair comparison, we use polynomial degrees two (first row), three (second row), and four (third row) on the same three (first row), four (second row), and five (third row) LGL nodes per element as the DG-USBP method. 
That is, we use the same set of points for both the DGSEM and the DG-USBP schemes.

\subsubsection*{Discussion of the results}

Based on the above results, we observe that the accuracies of the DGSEM and DG-USBP approaches are comparable, particularly for small values of the parameter $\lambda$, although the DGSEM consistently yields more accurate results.
We suspect that the reduction in accuracy for the DG-USBP method is due to the additional artificial dissipation introduced.
Therefore, while combining flux vector splitting techniques with DG-USBP operators does not necessarily lead to excessive artificial dissipation, the accuracy of the DG-USBP method remains lower than that of the DGSEM considered here.

\subsection{Spectral analysis}
\label{sub:spectral}

\begin{figure}[tb]
	\centering
	\begin{subfigure}[b]{0.45\textwidth}
		\includegraphics[width=\textwidth]{%
      		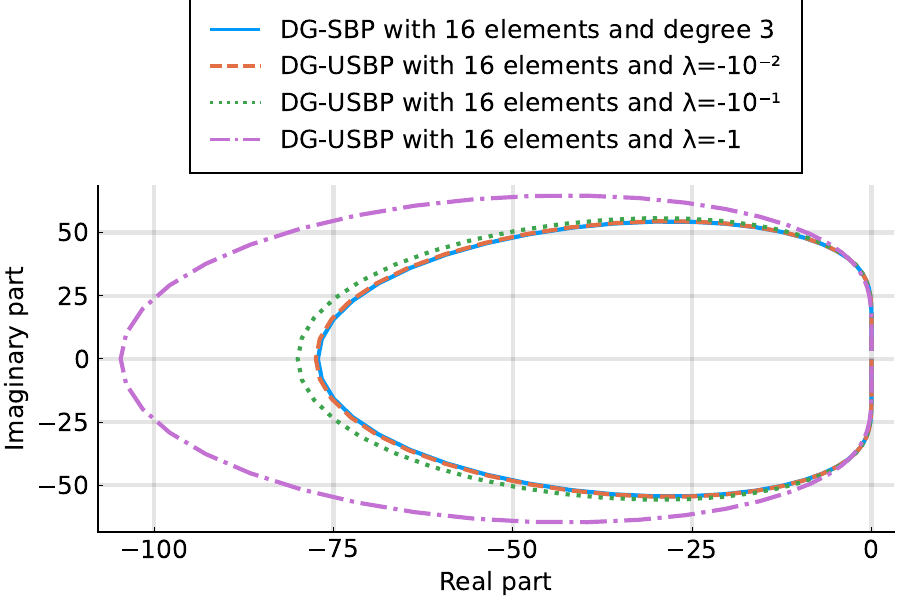}
    		\caption{Degree two DG-USBP operators}
    		\label{fig:spectra_N4_USBP_vs_DGSEM_16elements}
  	\end{subfigure}%
	\begin{subfigure}[b]{0.45\textwidth}
		\includegraphics[width=\textwidth]{%
      		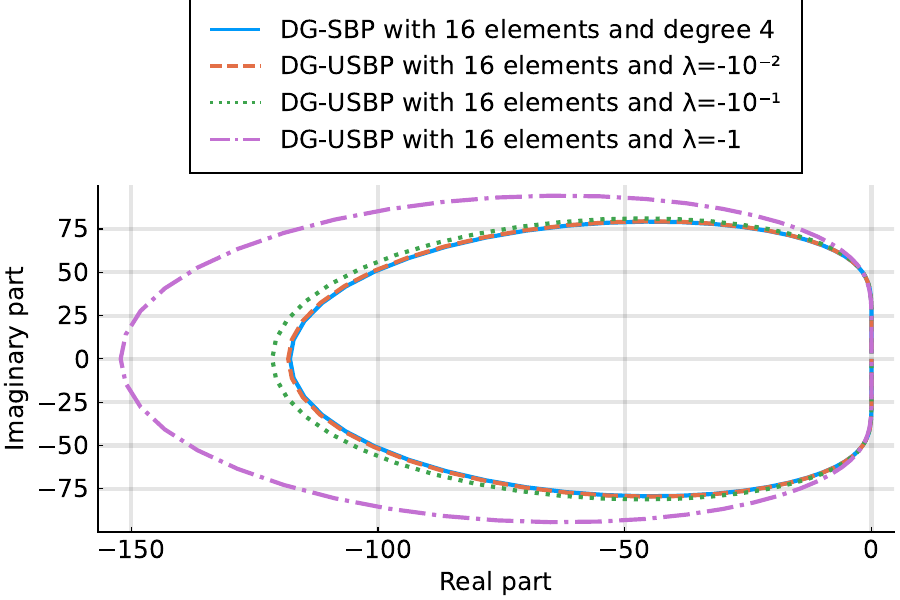}
    		\caption{Degree three DG-USBP operators}
    		\label{fig:spectra_N5_USBP_vs_DGSEM_16elements}
  	\end{subfigure}%
  	\caption{
  	Spectra of the DG-USBP and the DGSEM semi-discretizations with $J=16$ elements.
	We use degree two ($d=2$, \cref{fig:spectra_N4_USBP_vs_DGSEM_16elements}) and three ($d=3$, \cref{fig:spectra_N5_USBP_vs_DGSEM_16elements}) DG-USBP operators on four ($N=4$) and five ($N=5$) LGL nodes, respectively, with different parameters $\lambda_4$ and $\lambda_5$.
	For the DGSEM method, we use traditional degree three and four DG-SBP operators on the same LGL nodes as the respective DG-USBP operators.
  	}
  	\label{fig:spectra_USBP_vs_DGSEM}
\end{figure}

Next, we compare the spectra of the DG-USBP to the DGSEM and FD-USBP semi-discretization for the linear advection equation \cref{eq:linear_adv_tests} with periodic boundary conditions.
\cref{fig:spectra_USBP_vs_DGSEM} illustrates the spectra of the DG-USBP semi-discretization for degree two ($d=2$) and three ($d=3$) DG-USBP operators on four ($N=4$) and five ($N=5$) LGL nodes, respectively, with $J=16$ elements and different parameters $\lambda_4$ and $\lambda_5$.
For $\lambda_4$ and $\lambda_5$ close to zero, the DG-USBP spectra are comparable to the DGSEM method (labeled ``DG-SBP") on the same LGL points.
At the same time, for smaller $\lambda_4$ and $\lambda_5$, the spectra suggest that the DG-USBP method is stiffer than the DGSEM method.

\subsubsection*{Discussion of the results}

We find that the DG-USBP method is stiffer than the DGSEM method on the same LGL points in all cases.
This could be explained by the DG-USBP method having more artificial dissipation than the DGSEM method.
A drawback of increasing stiffness when using DG-USBP operators is that explicit time-stepping may require smaller time steps, which can impact the method's overall performance.

\subsection{Local linear/energy stability}
\label{sub:local_syability}

We now demonstrate the local linear/energy stability properties of the proposed DG-USBP operators for the inviscid Burgers' equation
\begin{equation}\label{eq:Burgers}
	\partial_t u(t,x) + \partial_x \left[ u(t,x)^2/2 \right] = 0, \quad
	x \in (-1,1),
\end{equation}
with periodic boundary conditions.
To this end, we linearize \cref{eq:Burgers} locally around a baseflow $\tilde{u}(x)$, so that $u(t,x) = \tilde{u}(x) + v(t,x)$, which yields
\begin{equation}\label{eq:Burgers_linear}
	\partial_t v(t,x) + \partial_x \left[ \tilde{u}(x) v(t,x) \right] = 0, \quad
	x \in (-1,1).
\end{equation}
Observe that \cref{eq:Burgers_linear} is a linear advection equation with a spatially varying coefficient $\tilde{u}(x)$ that is solved for the perturbation component $v(t,x)$.
Notably, when the baseflow $\tilde{u}(x)$ is positive across $[-1,1]$, the operator $v \mapsto \partial_x \left[ \tilde{u} v \right]$ becomes skew-symmetric for a weighted $L^2$ inner product.
As a result, this operator exhibits a purely imaginary spectrum  \cite{manzanero2018insights,gassner2022stability}.
A semi-discretization should replicate this behavior, ensuring a linearization around a positive baseflow yields eigenvalues with non-positive real parts.
However, it was observed in \cite{gassner2022stability} that existing entropy-stable high-order flux differencing DGSEM semi-discretizations can produce a linearized discrete operator that has eigenvalues with a significant positive real part.
Now consider the full-upwind DG-USBP semi-discretization,
\begin{equation}\label{eq:Burgers_DGUSBP}
 	\frac{\d}{\d t} \mathbf{u}^{j} + D_- \mathbf{f}^{j} = \mathbf{SAT}^j, \quad j=1,\dots,J,
\end{equation}
where $J$ is the number of elements and $\mathbf{f}^j = \frac{1}{2} \mathbf{u}^j \odot \mathbf{u}^j$ with $\odot$ denoting the usual element-wise Hadamard product.
Our numerical investigations demonstrate that the full-upwind DG-USBP semi-discretization, as per \cref{eq:Burgers_DGUSBP}, results in a linearized discrete operator whose eigenvalues have non-positive real parts, by the analysis of \cite{ranocha2023high}.
To stress the method, we consider a completely under-resolved case by computing the Jacobian at a random non-negative state using automatic differentiation via ForwardDiff.jl \cite{revels2016forward}.
\cref{tab:linear_stability} reports the maximum real parts of the spectrum of the DG-USBP method for the linearized Burgers equation \cref{eq:Burgers_linear}.
This is done using degree one ($d=1$), two ($d=2$), and three ($d=3$) DG-USBP operators on three ($N=3$), four ($N=4$), and five ($N=5$) LGL nodes per element, respectively, varying the number of elements $J$, and different values for the artificial dissipation parameters $\lambda_{3}$, $\lambda_{4}$, and $\lambda_{5}$.
In all cases, the maximum real part of the eigenvalues is around the machine precision level, aligning with our desired outcome.

\begin{table}[tb]
	\centering
	\begin{adjustbox}{width=0.99\textwidth}
	\begin{tabular}{r | r r r r}
  		\multicolumn{4}{c}{Degree one operator} \\
		\toprule
		$J$\textbackslash$\lambda_3$ & $0$ & $-10^{-2}$ & $-1$ \\
		\midrule
    		2 &  -1.5e-18 &   1.2e-16 &   9.9e-17 \\
   		4 &  -2.7e-16 &  -6.2e-16 &   2.9e-17 \\
   		8 &   4.3e-16 &  -4.0e-16 &   6.5e-16 \\
  		16 &  -9.9e-16 &   2.3e-16 &  -3.6e-15 \\
  		32 &   7.1e-18 &   8.8e-15 &   1.1e-14 \\
		\bottomrule
	\end{tabular}
	\hspace{.05cm}
  	\begin{tabular}{r | r r r r}
  		\multicolumn{4}{c}{Degree two operator} \\
		\toprule
		$J$\textbackslash$\lambda_4$ & $0$ & $-10^{-2}$ & $-1$ \\
		\midrule
    		2 &   6.7e-16 &   4.9e-17 &  -3.3e-16 \\
   		4 &  -4.9e-16 &   6.0e-16 &   5.5e-17 \\
   		8 &  -5.8e-16 &   2.2e-15 &  -4.1e-17 \\
  		16 &   3.3e-17 &  -1.1e-15 &  -3.5e-16 \\
  		32 &   1.0e-14 &  -4.3e-14 &   3.9e-15 \\
		\bottomrule
	\end{tabular}
	\hspace{.05cm}
	\begin{tabular}{r | r r r r}
  		\multicolumn{4}{c}{Degree three operator} \\
		\toprule
		$J$\textbackslash$\lambda_5$ & $0$ & $-10^{-2}$ & $-1$ \\
		\midrule
    		2 &   5.8e-16 &  -8.4e-16 &   7.8e-17 \\
   		4 &   1.5e-17 &  -5.3e-16 &   4.7e-16 \\
   		8 &   1.0e-15 &   3.6e-16 &  -1.8e-16 \\
  		16 &  -7.0e-16 &  -7.4e-15 &  -1.7e-15 \\
  		32 &   2.0e-15 &  -3.5e-16 &  -3.2e-15 \\
		\bottomrule
	\end{tabular}
	\end{adjustbox}
	\caption{
		The maximum real parts of the spectrum of the DG-USBP method for the Burgers equation using a full-upwind semi-discretization.
		 We use degree one ($d=1$), two ($d=2$), and three ($d=3$) USBP operators on three ($N=3$), four ($N=4$), and five ($N=5$) LGL nodes per element, respectively, varying numbers of elements $J$, and different values for the artificial dissipation parameter $\lambda_3$, $\lambda_4$, and $\lambda_5$.
	}
  	\label{tab:linear_stability}
\end{table}

\subsubsection*{Discussion fo the results}

We find that DG-USBP methods yield local linear/energy stability.
This finding contrasts the behavior observed for existing entropy-stable high-order flux differencing DGSEM semi-discretizations in the same setting.
Parallel findings were obtained in \cite{ranocha2023high} for traditional FD-USBP operators, complemented by a theoretical analysis.

\subsection{Curvilinear meshes and free-stream preservation}
\label{sub:fsp_tests}

One particular strength of DG-type methods is their geometric flexibility and ability to create high-order approximations on curvilinear, unstructured meshes.
This geometric flexibility is retained with the DG-USBP operators developed in the present work, which we subsequently demonstrate in a two-dimensional setting.

\subsubsection*{The curvilinear DG-USBP semi-discretization}

We start by describing conservation laws in curvilinear coordinates and their DG-USBP semi-discretization.
See Kopriva~\cite{kopriva2006metric} for a complete discussion on curvilinear coordinates and transformations. 
Recall that a conservation law in two dimensions is of the form
\begin{equation}\label{eq:fsp_CL}
	\partial_t u + \partial_x f_1( u ) + \partial_y f_2( u ) = 0, \quad
	t \in (0,T), \ (x,y) \in \Omega \subset \R^2,
\end{equation}
with fluxes fluxes $f_1$, $f_2$, and conserved variable $u = u(t,x,y)$.
Suppose the domain $\Omega$ has been subdivided into $J$ non-overlapping quadrilateral elements, $\Omega_j$, $j=1,\dots,J$.
For ease of notation, we subsequently consider the conservation law \cref{eq:fsp_CL} on an individual element $\Omega_j$ and suppress the index $j$.
The element $\Omega_j$ in the physical coordinates $(x,y)$ is transformed into a reference element $\Omega_{\rm ref} = [-1,1]^2$ in the computational coordinates $(\xi,\eta)$ via the coordinate transformation
\begin{equation}
	x = X(\xi,\eta), \quad
	y = Y(\xi,\eta).
\end{equation}
Under this transformation, the conservation law in the physical coordinates \cref{eq:fsp_CL} becomes a conservation law in the reference coordinates:
\begin{equation}\label{eq:fsp_CL_trans}
	J \partial_t u + \partial_{\xi} \tilde{f}_1( u ) + \partial_{\eta} \tilde{f}_2( u ) = 0, \quad
	t \in (0,T), \ (\xi,\eta) \in \Omega_{\rm ref},
\end{equation}
with $u = u(t,\xi,\eta)$.
The contravariant fluxes are $\tilde{f}_1 = Y_{\eta} f_1 - X_{\eta} f_2$ and $\tilde{f}_2 = -Y_{\xi} f_1 + X_{\xi} f_2$, while the Jacobian is $J = X_{\xi} Y_{\eta}  - X_{\eta} Y_{\xi}$.
The partial derivatives $X_{\xi}$, $X_{\eta}$, $Y_{\xi}$, and $Y_{\eta}$, are called the \emph{metric terms}.
The metric terms in DG approximations are typically created from a transfinite interpolation with linear blending \cite{kopriva2006metric,gordon1973construction}.
The curvilinear DG-USBP semi-discretization of \cref{eq:fsp_CL_trans} has the form
\begin{equation}\label{eq:transformed_split}
	{J} \partial_t \mathbf{u} + D_- \tilde{\mathbf{f}}^{\,+}_1 + D_+ \tilde{\mathbf{f}}^{\,-}_1
	+ \tilde{\mathbf{f}}^{\,+}_2 D_-^T + \tilde{\mathbf{f}}^{\,-}_2 D_+^T
	= \widetilde{\mathbf{SAT}}^j.
\end{equation}
Note that multiplication from the left with $D_{\pm}$ approximates the derivative in the $\xi$-direction
and multiplication from the right with $D_{\pm}^T$ approximates the derivative in the $\eta$-direction.
Furthermore, we use a generic statement of the $\mathbf{SAT}$ in the normal direction on an interface in element $j$
\begin{equation}
	\widetilde{\mathbf{SAT}}^j
		= - P^{-1} \mathbf{e}_R \left( \fnum(\mathbf{u}^j_R, \mathbf{u}^{j+1}_L; \hat{\mathbf{n}}) - \tilde{\mathbf{f}}^{\,j}_R \right)
  		+ P^{-1} \mathbf{e}_L \left( \fnum(\mathbf{u}^{j-1}_R, \mathbf{u}^j_L; \hat{\mathbf{n}}) - \tilde{\mathbf{f}}^{\,j}_L \right),
\end{equation}
where $\hat{\mathbf{n}}$ is the normal direction on a particular element.
The tilde notation indicates that the physical fluxes are computed in the contravariant direction.

\subsubsection*{Free-stream preservation}

Notably, the metric terms satisfy the two metric identities
\begin{equation}\label{eq:fsp_metric_identities}
	\partial_{\xi} Y_{\eta} - \partial_{\eta} Y_{\xi} = 0, \quad
	\partial_{\xi} X_{\eta} - \partial_{\eta} X_{\xi} = 0,
\end{equation}
which are essential to ensure \emph{free-stream preservation (FSP)}; see \cite{visbal1999high,vinokur2002extension,kopriva2006metric}.
That is, given a constant flux in space, its divergence vanishes, and the constant solution of \cref{eq:fsp_CL} does not change in time.
As demonstrated in \cite{ranocha2023high}, the ability of existing FD-USBP methods to remain FSP on curvilinear meshes is delicate.
There is a subtle relationship between the boundary closure accuracy of an FD-SBP operator, the dependency of flux vector splitting on the metric terms, and the polynomial degree of curved boundaries.
The boundary closure of an FD-SBP operator relates to the highest polynomial degree the operator can differentiate exactly.
For instance, a 4-2 FD-SBP operator is fourth-order accurate in the interior with second-order boundary closures and can differentiate up to quadratic polynomials exactly.
In the following discussion, we use the term boundary closure in the DG context to refer to the highest polynomial degree a discrete derivative operator can differentiate exactly.

The curvilinear flux vector splittings $\tilde{\mathbf{f}}^{\pm}_1$ and $\tilde{\mathbf{f}}^{\pm}_2$ in \cref{eq:transformed_split} implicitly depend on the metric terms \cite{anderson1986comparison,ranocha2023high}.
The polynomial dependency of the curvilinear flux vector splitting on the metric terms must ``agree'' with the boundary order closure of a USBP operator.
Precisely, suppose the polynomial dependency of a mapped flux vector splitting on the metric terms has order $m$. In that case, the USBP operator must have a boundary closure that can differentiate polynomials up to degree $mN_{\textrm{geo}}$ exactly, where $N_{\textrm{geo}}$ is the polynomial order of the curved boundaries.
For example, the Lax--Friedrichs splitting is linear in the metric terms while the van Leer-H\"{a}nel splitting is quadratic in the metric terms.
This means that $m=1$ for the Lax--Friedrichs splitting and the curvilinear approximation remains FSP regardless of the boundary closure order.
In contrast, the van Leer-H\"{a}nel splitting has $m=2$ (quadratic dependency on the metric terms), which restricts the possible values of $N_{\textrm{geo}}$ combined with the boundary closure of the SBP operator.

\begin{remark}
To create DG-USBP operators for polynomials of degree $d$, we sacrifice one degree of accuracy on the boundary closure to construct the internal dissipation matrix $S$, i.e., we use $N=d+2$ LGL points.
A ``classic'' DG-SBP operator with $N=4$ LGL points has a boundary closure of the order $d=3$ and can differentiate up to cubic polynomials exactly.
In contrast, a DG-USBP operator for $N=4$ LGL points will be endowed with a boundary closure of degree $d=2$ and can only differentiate up to quadratic polynomials exactly.
\end{remark}

\subsubsection*{Computational results}

\begin{figure}[tb]
	\centering
  	\begin{subfigure}{0.495\textwidth}
  		\centering
    		\includegraphics[width=\textwidth]{%
			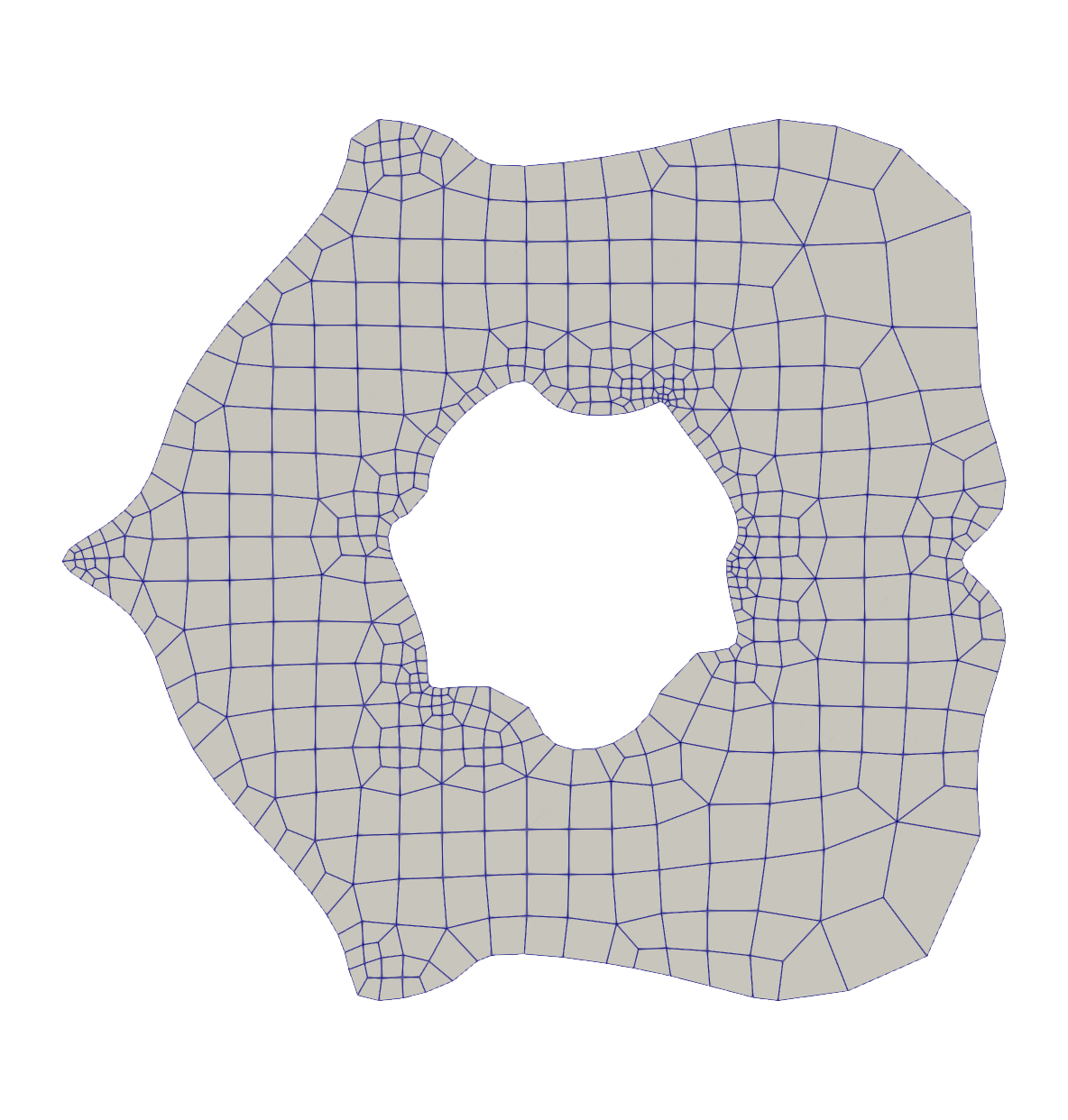}
    		\caption{Linear boundaries}
  	\end{subfigure}%
  	\begin{subfigure}{0.495\textwidth}
  		\centering
    		\includegraphics[width=\textwidth]{%
			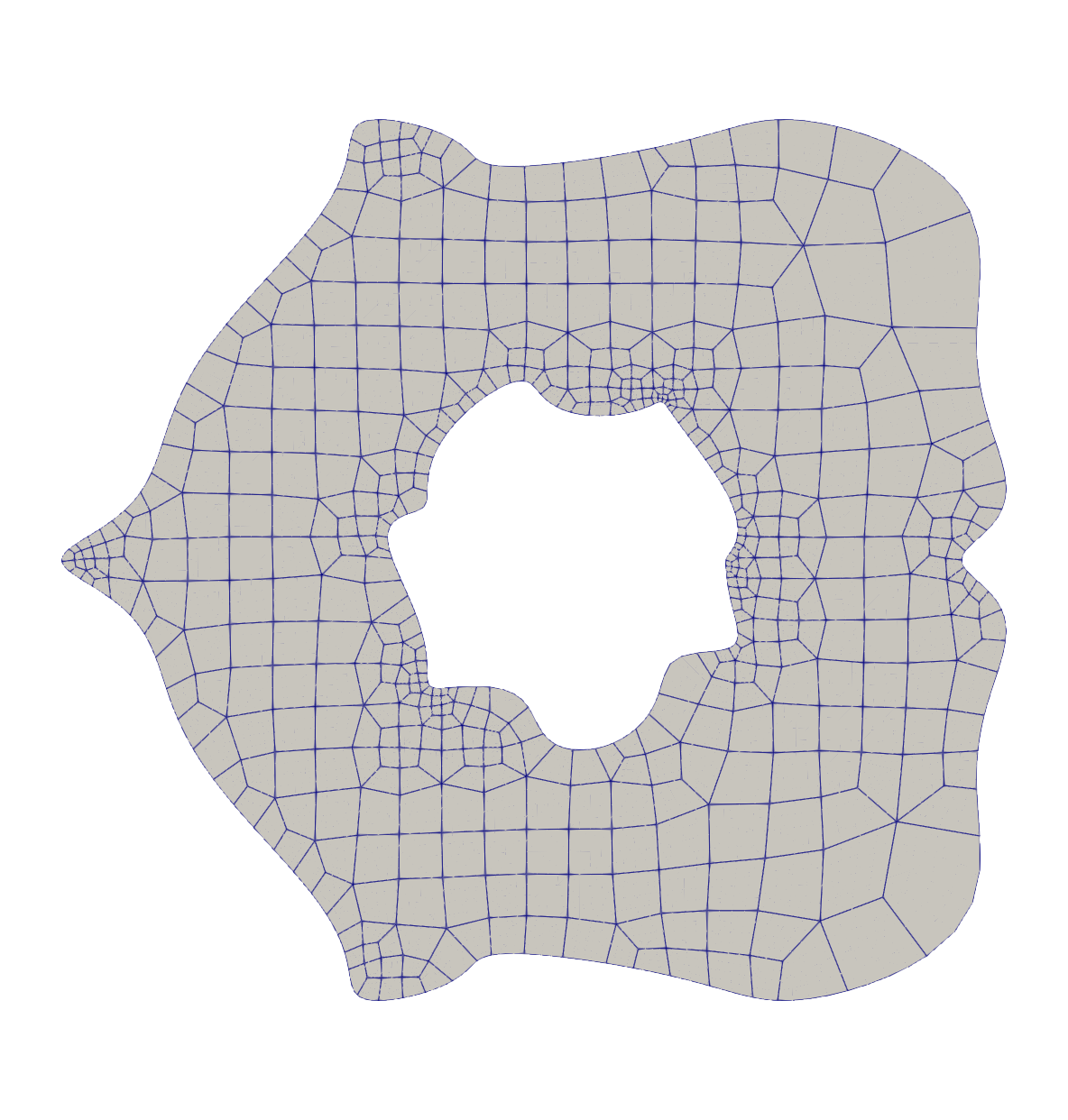}
    		\caption{Quartic boundaries}
  	\end{subfigure}%
  	\caption{
		Meshes of 550 non-overlapping quadrilaterals with linear and quartic boundary polynomial order for the FSP test.
		Observe---see the lower right corner---that the boundary is approximated more accurately for the higher (quartic) boundary polynomial.
	}
  	\label{fig:fsp_meshes}
\end{figure}

We now numerically investigate how the type of flux vector splitting in curvilinear coordinates combined with the DG-USBP operator influences the FSP property.
To this end, we consider an unstructured curvilinear mesh with varying $N_{\textrm{geo}}$ values and present the errors for an FSP problem for different flux vector splittings and operator orders.
Specifically, we consider a domain with curvilinear outer and inner boundaries of different geometric polynomial degrees, generated in Julia with HOHQMesh.jl~\cite{kopriva2024hohqmeshjl}.
The domain is illustrated in \cref{fig:fsp_meshes} for linear and quartic boundaries.
The curvilinear domain is a heavily warped annulus-type domain with the outer and inner curved boundaries given by
\begin{equation}
\begin{aligned}
	\textrm{outer}
		&= \begin{cases}
			x(t) = 4\cos(2 \pi t) - 0.6 \cos^3(8 \pi t) \\
			y(t) = 4\sin(2 \pi t) - 0.5 \sin(10 \pi t)
		\end{cases}	\\
	\textrm{inner}
		&= \begin{cases}
			x(t) = 1.5 \cos(2 \pi t) - 0.2 \cos^2(8 \pi t)\\
			y(t) = 1.5 \sin(2 \pi t) - 0.2 \sin(11 \pi t)
		\end{cases}
\end{aligned}
\end{equation}
with the parameter $t\in[0,1]$.
The domain is divided into 550 non-overlapping quadrilateral elements.
The domain boundaries are highly nonlinear, and we compute results on four versions of the approximate mesh boundaries:
bi-linear boundary elements ($N_{\textrm{geo}}=1$),
quadratic boundary elements ($N_{\textrm{geo}}=2$),
cubic boundary elements ($N_{\textrm{geo}}=3$), and
quartic boundary elements ($N_{\textrm{geo}}=4$).
\cref{fig:fsp_meshes} shows the bilinear and quartic boundary meshes to compare
the resolution of the curved boundary elements as $N_{\textrm{geo}}$ increases. 
Furthermore, we consider the two-dimensional compressible Euler equations on the curvilinear domain illustrated in \cref{fig:fsp_meshes} with a constant solution state given in the conservative variables as
\begin{equation}\label{eq:fsp_sol}
	u_\infty
	=
	\begin{pmatrix}
		\rho_\infty\\
		(\rho v_1)_\infty\\
		(\rho v_2)_\infty\\
		(\rho e)_\infty
	\end{pmatrix}
	=
	\begin{pmatrix}
		1.0\\
		0.1\\
		-0.2\\
		10.0
	\end{pmatrix},
\end{equation}
where the compressible Euler fluxes are all constant, and their divergence vanishes.
Importantly, this is not necessarily the case for the numerical solution.
To demonstrate this, we examine the FSP property for the Lax--Friedrichs and van Leer-H\"{a}nel splittings for different values of $N_{\textrm{geo}}$ and varying orders of the proposed DG-USBP operators.
See \cref{tab:fsp_LLF,tab:fsp_vLH}, respectively.
In all cases, the final time of all the FSP tests is $10$ and
we impose Dirichlet boundary conditions via the background solution state \cref{eq:fsp_sol}.

\begin{table}[tb]
\centering
   	\begin{adjustbox}{max width=\textwidth}
    \begin{tabular}{lcccccccccc}
      	\toprule
      	interior order & 2 & 3 & 4 & 5 & 6 & 7 & 8 & 9 \\
      	\midrule
      	bi-linear mesh & \num{2.92e-14} & \num{9.96e-15} & \num{1.06e-14} & \num{2.08e-14} & \num{1.83e-14} & \num{3.76e-14}  & \num{4.12e-14} & \num{5.96e-14} \\
      	quadratic mesh & \num{1.19e-14} & \num{1.12e-14} & \num{5.53e-14} & \num{1.93e-14} & \num{1.81e-14} & \num{4.12e-14}  & \num{3.89e-14} & \num{2.84e-14} \\
      	cubic mesh & \num{5.52e-14} & \num{1.45e-14} & \num{7.49e-15} & \num{1.92e-14} & \num{2.05e-14} & \num{6.67e-14} & \num{4.56e-14} & \num{5.67e-14} \\
      	quartic mesh & \num{2.29e-15} & \num{8.95e-15} & \num{6.03e-15} & \num{2.00e-14} & \num{2.16e-14} & \num{5.83e-14} & \num{2.37e-14} & \num{4.81e-14} \\
      	\bottomrule
   	\end{tabular}
    \end{adjustbox}
    \caption{
    		FSP error for the Lax--Friedrichs splitting.
		The splitting depends linearly on the metric terms and satisfies the FSP property for all approximation and mesh orders.
	}
	\label{tab:fsp_LLF}
\end{table}

\begin{table}[tb]
\centering
   	\begin{adjustbox}{max width=\textwidth}
    \begin{tabular}{lcccccccc}
      	\toprule
      	interior order & 2 & 3 & 4 & 5 & 6 & 7 & 8 & 9\\
      	\midrule
      	bi-linear mesh & \num{1.21e-5} & \num{4.44e-15} & \num{6.49e-15} & \num{1.20e-14} & \num{1.40e-14} & \num{3.85e-14} & \num{5.31e-13} & \num{9.56e-13} \\
      	quadratic mesh & \num{1.38e-5} & \num{1.43e-6} & \num{1.81e-7} & \num{1.11e-14} & \num{1.51e-14} & \num{1.15e-14} & \num{2.22e-13} & \num{7.56e-13} \\
      	cubic mesh & \num{1.37e-5} & \num{1.56e-6} & \num{2.28e-7} & \num{2.40e-8} & \num{2.42e-9} & \num{3.09e-14} & \num{3.75e-13} & \num{9.19e-13} \\
      	quartic mesh & \num{1.38e-5} & \num{1.57e-6} & \num{3.93e-7} & \num{4.21e-8} & \num{1.05e-8} & \num{9.93e-10} & \num{2.21e-10} & \num{9.87e-13} \\
      	\bottomrule
    	\end{tabular}
    \end{adjustbox}
    \caption{
    		FSP error for the van Leer--H\"{a}nel splitting.
		The splitting depends quadratically on the metric terms and satisfies the FSP property if the boundary closure accuracy of the DG-USBP operator is not smaller than $2N_{\textrm{geo}}$.
	}
	\label{tab:fsp_vLH}
\end{table}

\subsubsection*{Discussion of the results}

In accordance with \cite{ranocha2023high}, the Lax--Friedrichs splitting, depending linearly on the metric terms, satisfies the FSP property for all considered orders.
In contrast, the van Leer--H\"{a}nel splitting, depending quadratically on the metric terms, only satisfies the FSP property if the boundary closure of the DG-USBP operator can exactly differentiate polynomials up to order $mN_{\textrm{geo}} = 2 N_{\textrm{geo}}$.
Finally, our results in \cref{tab:fsp_vLH} imply that more complex flux vector splittings require higher boundary closure accuracy to ensure the boundary is sufficiently accurate.

\subsection{Isentropic vortex}
\label{sub:isentropic_vortex}

Next, we investigate the long-term stability of the DG-USBP method.
To this end, we follow \cite{sjogreen2018high} and consider the classical isentropic vortex test case of \cite{shu1997essentially} for the two-dimensional Euler equations with initial conditions
\begin{equation}
\begin{aligned}
  	T & = T_0 - \frac{(\gamma-1) \epsilon^2}{8 \gamma \pi^2} \exp\bigl(1-r^2\bigr), \\
   	\rho & = \rho_0 (T / T_0)^{1 / (\gamma - 1)}, \\
	v & = v_0 + \frac{\varepsilon}{2 \pi} \exp\bigl((1-r^2) / 2\bigr) (-x_2, x_1)^T.
\end{aligned}
\end{equation}
Here, $\epsilon = 10$ is the vortex strength, $r$ is the distance from the origin, $T = p / \rho$ the temperature, $\rho_0 = 1$ the background density, $v_0 = (1, 1)^T$ the background velocity,
$p_0 = 10$ the background pressure, $\gamma = 1.4$, and
$T_0 = p_0 / \rho_0$ the background temperature.
The domain $[-5, 5]^2$ is equipped with periodic boundary conditions.
We use the same time integration method and approach to compute the discrete $L^2$ error of the density as in \Cref{sub:convergence_Euler}.

\begin{figure}[tb]
	\centering
	\begin{subfigure}[b]{0.495\textwidth}
		\includegraphics[width=\textwidth]{%
      		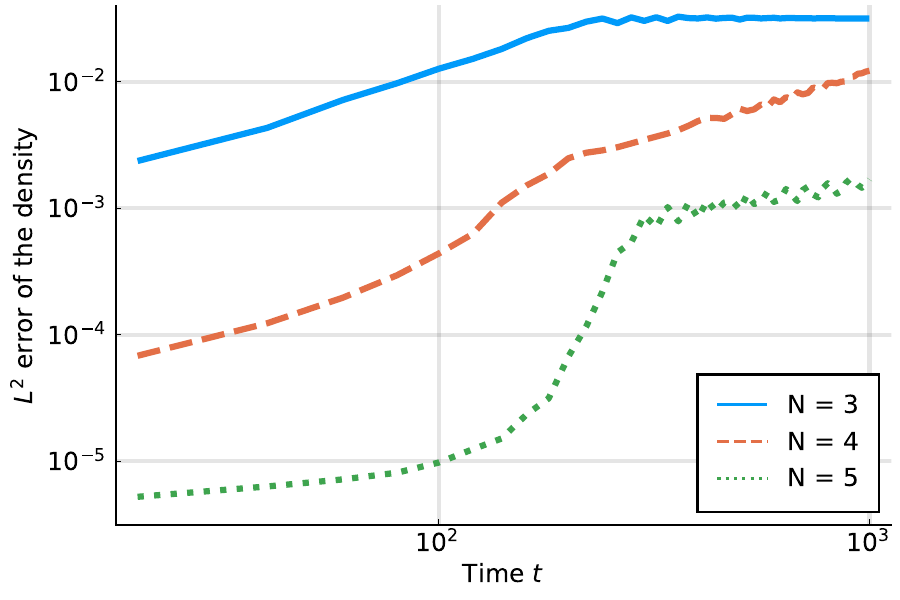}
    		\caption{DGSEM}
	\label{fig:isentropic_vortex_DGSEM_256elements}
  	\end{subfigure}%
	\begin{subfigure}[b]{0.495\textwidth}
		\includegraphics[width=\textwidth]{%
      		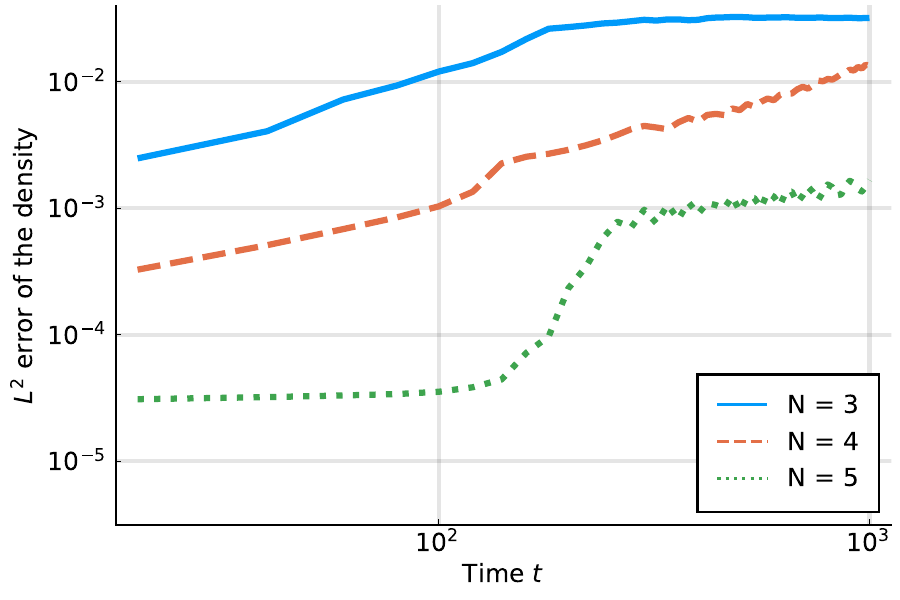}
    		\caption{DG-USBP method with $\lambda_N = -10^{-3}$}
		\label{fig:isentropic_vortex_USBP_256elements_sigma3}
  	\end{subfigure}%
  	\caption{
	Density error for the isentropic vortex test case of the DG-USBP and DGSEM method for $J=256$ elements.
  	We use degree one ($d=1$, blue dashed line), two ($d=2$, orange dotted line), and three ($d=3$, green dash-dotted line) tensor-product DG-USBP operators on $N=3, 4, 5$ LGL nodes per coordinate direction and element.
	For the DGSEM method, we use traditional degree two, three, and four DG-SBP operators on the same LGL nodes as the DG-USBP operators.
  	}
  	\label{fig:isentropic_vortex}
\end{figure}

\cref{fig:isentropic_vortex} illustrates the discrete $L^2$-error of the density for long-time simulations for the DG-USBP method with Steger--Warming splitting and the DGSEM using the volume flux of Ranocha.
We use three ($N=3$, blue straight line), four ($N=4$, orange dashed line), and five ($N=5$, green dotted line) LGL nodes per coordinate direction and element.
As before, the DG-USBP and DGSEM methods use the same points.
\cref{fig:isentropic_vortex} report these results for $J= 256$ elements and $\lambda_N = -10^{-3}$.
The remaining parameters are zero to preserve exactness for polynomials up to degree $d=N-2$.
We observe that both the DG-USBP and DGSEM methods remain stable and can run the simulations successfully for a long time.
Notably, the two methods perform roughly the same for three LGL nodes.
For four and five LGL nodes, however, we observe the DGSEM method to initially produce smaller density errors than the DG-USBP method.
In all cases, the two methods perform roughly the same towards the end of the simulation, indicating a comparable long-time behavior.

\subsubsection*{Discussion of the results}

The above results indicate that USBP operators offer less improvement in robustness for DG methods compared to FD schemes.

\subsection{Kelvin--Helmholtz instability}
\label{sub:KH_instability}

The above results for the isentropic vortex test (see \Cref{sub:isentropic_vortex}) indicated that USBP operators offer less improvement in robustness for DG methods compared to FD schemes.
To further investigate this hypothesis, we consider the Kelvin--Helmholtz instability setup for the two-dimensional compressible Euler equations of an ideal fluid as in  \cite{rueda2021}. 
The corresponding initial condition is
\begin{equation}
  \rho = \frac{1}{2} + \frac{3}{4} B(x,y),
  \quad
  p = 1,
  \quad
  v_1 = \frac{1}{2} \bigl( B(x,y) - 1 \bigr),
  \quad
  v_2 = \frac{1}{10}\sin(2 \pi x),
\end{equation}
where $B(x, y) = \tanh(15 y + 7.5) - \tanh(15 y - 7.5)$ is a smoothed approximation to a discontinuous step function.
The domain is $[-1,1]^2$ with time interval $[0, 15]$.
We perform time integration of the semi-discretizations using the third-order, four-stage SSP method detailed in \cite{kraaijevanger1991contractivity}.
This is coupled with the embedded method proposed in \cite{conde2022embedded} and an error-based step size controller as in \cite{ranocha2021optimized}.
The adaptive time step controller tolerances are set to $10^{-6}$.
Here, we compare the DG-USBP method outlined in this paper with the DGSEM, which employs various volume fluxes and a local Lax--Friedrichs (Rusanov) surface flux.

\begin{table}[tb]
\centering
	\begin{adjustbox}{width=0.99\textwidth}
	\begin{tabular}{ c | r r r r}
  		\multicolumn{5}{c}{$N=3$} \\
		\toprule
  		\multicolumn{1}{c}{$J$} & $16$ & $64$ & $256$ & $1024$ \\
		\midrule
		\multicolumn{5}{c}{DG-USBP, van Leer--H\"anel splitting} \\
		\midrule
		$\lambda_3$ & \multicolumn{4}{c}{final times} \\
		$-10^{-3}$
			& \textbf{15.00} & 4.93 & 1.89 & 2.00 \\
		$-10^{-2}$
			& \textbf{15.00} & \textbf{15.00} & 1.84 & 4.41 \\
		$-10^{-1}$
			& \textbf{15.00} & \textbf{15.00} & \textbf{15.00} & \textbf{15.00} \\
		\midrule
		\multicolumn{5}{c}{DG-USBP, Steger--Warming splitting} \\
    		\midrule
		$\lambda_3$ & \multicolumn{4}{c}{final times} \\
		$-10^{-3}$
			& \textbf{15.00} & 4.85 & 1.96 & 1.85 \\
		$-10^{-2}$
			& \textbf{15.00} & \textbf{15.00} & 1.88 & 4.30 \\
		$-10^{-1}$
			& \textbf{15.00} & \textbf{15.00} & \textbf{15.00} & \textbf{15.00} \\
		\midrule
		\multicolumn{5}{c}{DGSEM} \\
    		\midrule
		flux & \multicolumn{4}{c}{final times} \\
		Shima
			& \textbf{15.00} & 2.92 & 3.25 & 3.03 \\
		Ranocha
			& \textbf{15.00} & 4.68 & 4.81 & 4.12 \\
		\bottomrule
	\end{tabular}
	~
	\begin{tabular}{ c | r r r r}
  		\multicolumn{5}{c}{$N=4$} \\
		\toprule
  		\multicolumn{1}{c}{$J$} & $16$ & $64$ & $256$ & $1024$ \\
		\midrule
		\multicolumn{5}{c}{DG-USBP, van Leer--H\"anel splitting} \\
    		\midrule
		$\lambda_4$ & \multicolumn{4}{c}{final times} \\
		$-10^{-3}$
			& 4.57 & 10.73 & 1.60 & 1.85 \\
		$-10^{-2}$
			& 4.58 & 10.55 & 1.61 & 1.80 \\
		$-10^{-1}$
			& \textbf{15.00} & \textbf{15.00} & \textbf{4.82} & 3.48 \\
		\midrule
		\multicolumn{5}{c}{DG-USBP, Steger--Warming splitting} \\
    		\midrule
		$\lambda_4$ & \multicolumn{4}{c}{final times} \\
		$-10^{-3}$
			& 4.64 & 1.68 & 1.61 & 1.73 \\
		$-10^{-2}$
			& 4.62 & 1.78 & 1.62 & 1.81 \\
		$-10^{-1}$
			& \textbf{15.00} & 1.91 & 4.78 & 3.49 \\
		\midrule
		\multicolumn{5}{c}{DGSEM} \\
    		\midrule
		flux & \multicolumn{4}{c}{final times} \\
		Shima
			& 2.73 & 1.38 & 2.82 & 2.88 \\
		Ranocha
			& 4.46 & 1.53 & 3.77 & \textbf{3.66} \\
		\bottomrule
	\end{tabular}
	~
	\begin{tabular}{ c | r r r r}
  		\multicolumn{5}{c}{$N=5$} \\
		\toprule
  		\multicolumn{1}{c}{$J$} & $16$ & $64$ & $256$ & $1024$ \\
		\midrule
		\multicolumn{5}{c}{DG-USBP, van Leer--H\"anel splitting} \\
    		\midrule
		$\lambda_5$ & \multicolumn{4}{c}{final times} \\
		$-10^{-3}$
			& 1.38 & 2.91 & 1.48 & 3.62 \\
		$-10^{-2}$
			& 1.39 & 2.92 & 1.51 & 3.62 \\
		$-10^{-1}$
			& 1.52 & \textbf{6.23} & 3.69 & 3.61 \\
		\midrule
		\multicolumn{5}{c}{DG-USBP, Steger--Warming splitting} \\
    		\midrule
		$\lambda_5$ & \multicolumn{4}{c}{final times} \\
		$-10^{-3}$
			& 1.27 & 2.59 & 1.51 & 3.62 \\
		$-10^{-2}$
			& 1.29 & 2.88 & 1.52 & 3.64 \\
		$-10^{-1}$
			& 1.47 & 5.65 & 3.67 & 3.60 \\
		\midrule
		\multicolumn{5}{c}{DGSEM} \\
    		\midrule
		flux & \multicolumn{4}{c}{final times} \\
		Shima
			& 1.81 & 3.05 & 3.29 & 3.36 \\
		Ranocha
			& \textbf{2.47} & 4.04 & \textbf{4.44} & \textbf{4.27} \\
		\bottomrule
	\end{tabular}
	\end{adjustbox}
	\caption{
		Final times for the numerical simulations of the Kelvin--Helmholtz instability with $J$ elements.
		Final times less than 15 indicate that the simulation crashed.
		We use the same three ($N=3$), four ($N=4$), and five ($N=5$) LGL nodes per coordinate direction and element for the tensor-product DG-USBP and DGSEM method.
	We compare different parameters $\lambda_N$, splittings, and numerical fluxes.
	The highest final times for each combination of $J$ and $N$ are highlighted in bold.
	}
	\label{tab:kelvin_helmholtz}
\end{table}

\cref{tab:kelvin_helmholtz} presents the final simulation times for various numerical experiments of the Kelvin--Helmholtz instability.
Final times below 15 signify that the simulation terminated prematurely. 
For the tensor-product DG-USBP and DGSEM methods, we use the same three ($N=3$), four ($N=4$), and five ($N=5$) LGL nodes per coordinate direction and element.
To preserve exactness for polynomials up to degree $d=N-2$, we have chosen the other dissipation parameters as zero, i.e., $\lambda_1=\dots=\lambda_{N-1} = 0$, and test different parameters $\lambda_N$.
A notable observation is that reducing the parameters $\lambda_N$ in the DG-USBP method often leads to an extension in the final time of the corresponding simulations.
This trend suggests an enhancement in the method's robustness, although the difference decreases with an increased number of elements.
Specifically, for three LGL points ($N=3$) and sufficiently small $\lambda_3$, we observe the DG-USBP method to finish the simulation in all cases, while the DGSEM crashed for most of them.
At the same time, we observe that the DG-USBP method offers less improvement in robustness as the number of LGL points ($N$) increases.
For instance, the DGSEM with the Ranocha flux yields the longest run times in almost all cases for $N=5$.
Once more, this highlights that USBP operators offer less improvement in robustness for high-order DG methods than for FD schemes.

\subsubsection*{Discussion of the results}

We observe in \cref{tab:kelvin_helmholtz} that combining lower-order DG-USBP with flux vector splittings can increase the robustness of nodal DG methods.
Specifically, for three LGL points ($N=3$) and sufficiently small $\lambda_3$, we observe the DG-USBP method to finish the simulation in all cases, while the DGSEM crashed for most of them.
At the same time, we observe that the DG-USBP method offers less improvement in robustness as the number of LGL points ($N$) increases.
For instance, the DGSEM with the Ranocha flux yields the longest run times in almost all cases for $N=5$.
Once more, this highlights that USBP operators offer less improvement in robustness for high-order DG methods than for FD schemes.
This reduced improvement could be attributed to USBP methods adding dissipation only to unresolved modes.
After all, FD schemes typically have more unresolved modes than nodal DG methods.

\subsubsection*{Density profiles at crash time}

\begin{figure}[tb]
	\centering
	\begin{subfigure}[b]{0.45\textwidth}
		\includegraphics[width=\textwidth]{%
      		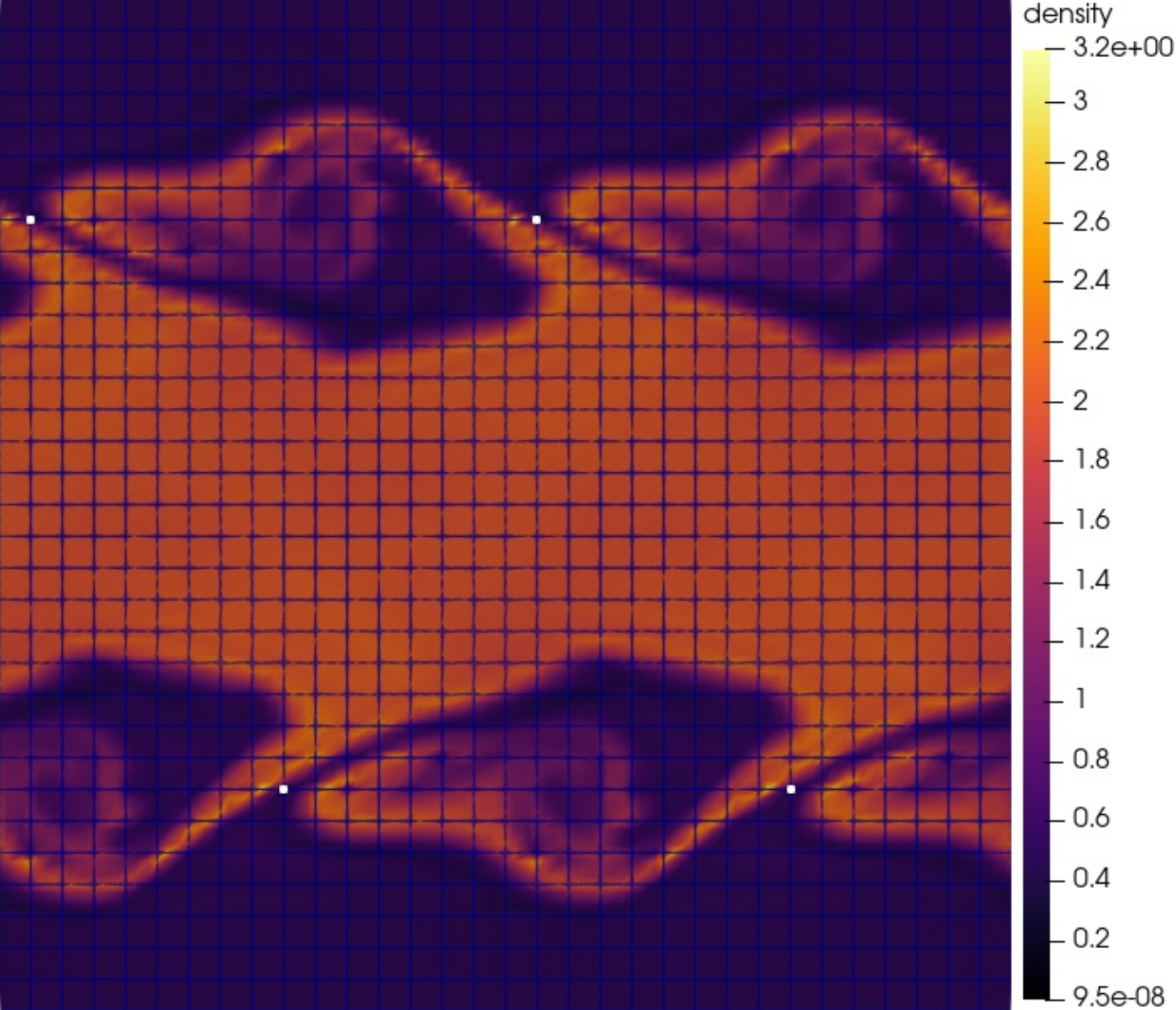}
    		\caption{DG-USBP, van Leer--H\"anel, crashed at $t=3.48$}
    		\label{fig:DGUSBP_vLH_N4_lambda_em1_1024elements}
  	\end{subfigure}%
	~
	\begin{subfigure}[b]{0.45\textwidth}
		\includegraphics[width=\textwidth]{%
      		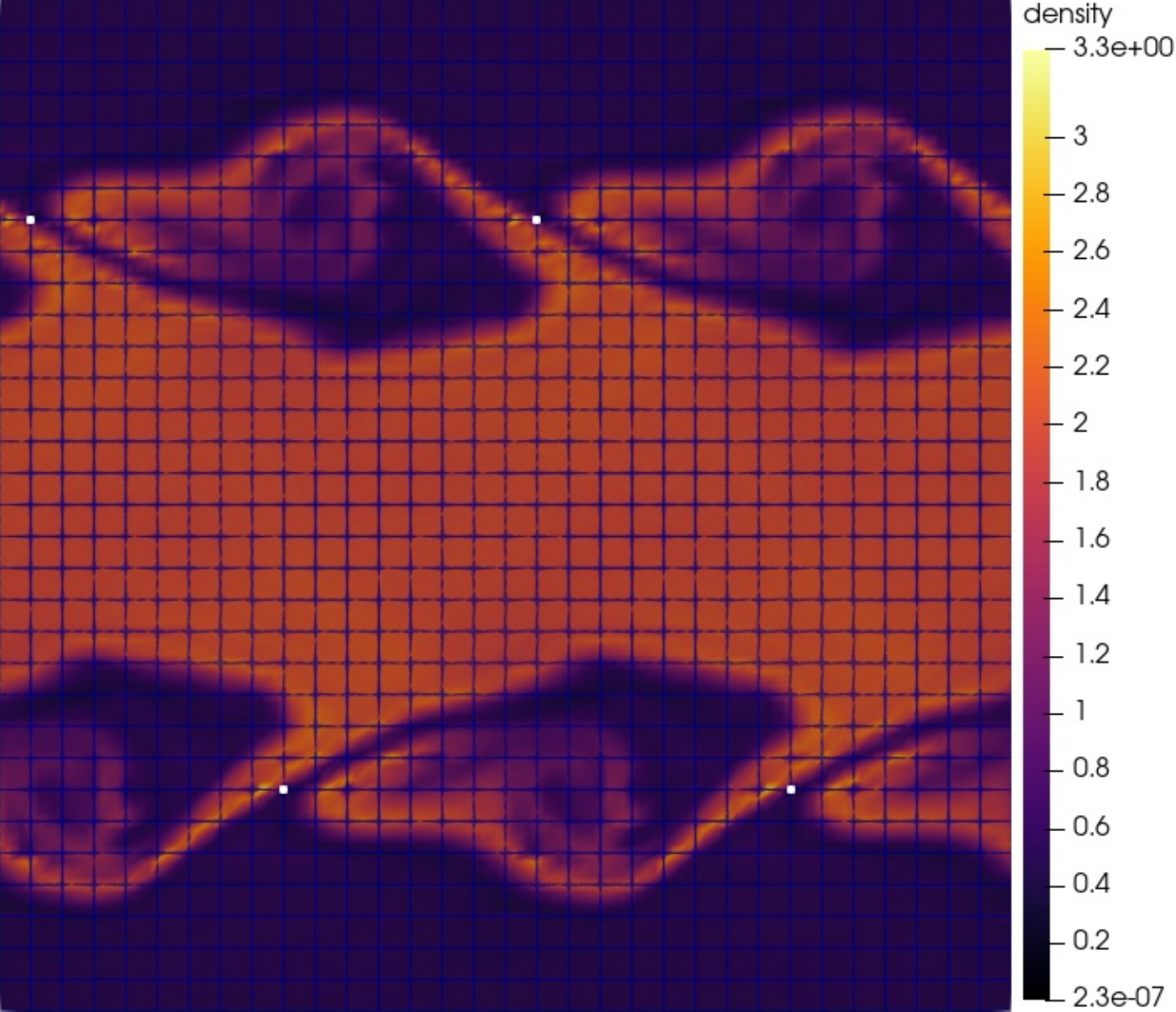}
    		\caption{DG-USBP, Steger--Warming, crashed at $t=3.49$}
    		\label{fig:DGUSBP_SW_N4_lambda_em1_1024elements}
  	\end{subfigure}%
	\\
	\begin{subfigure}[b]{0.45\textwidth}
		\includegraphics[width=\textwidth]{%
      		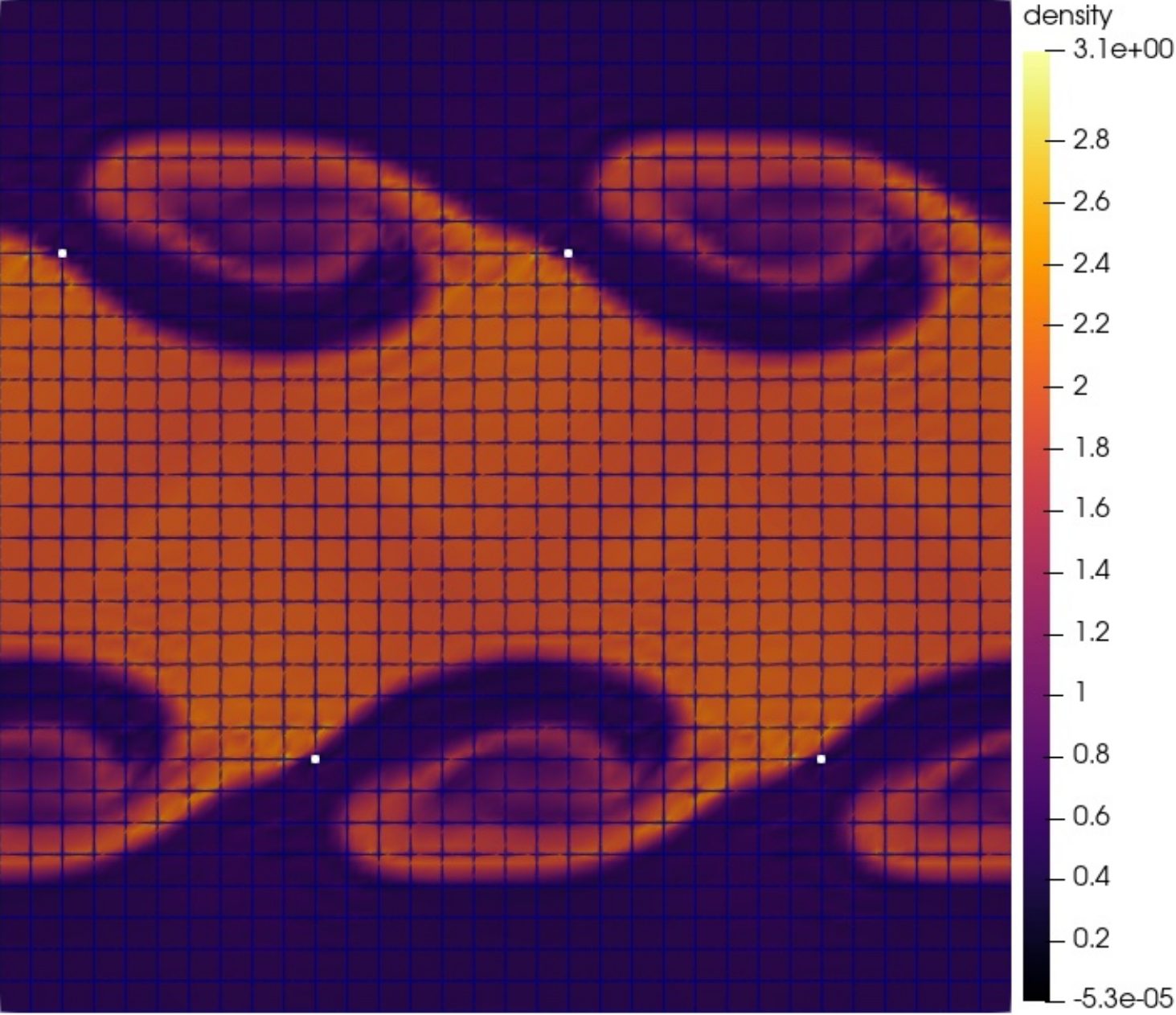}
    		\caption{DGSEM, Shima flux, crashed at $t=2.88$}
    		\label{fig:DGSEM_Shima_N4_1024elements}
  	\end{subfigure}%
	~
	\begin{subfigure}[b]{0.45\textwidth}
		\includegraphics[width=\textwidth]{%
      		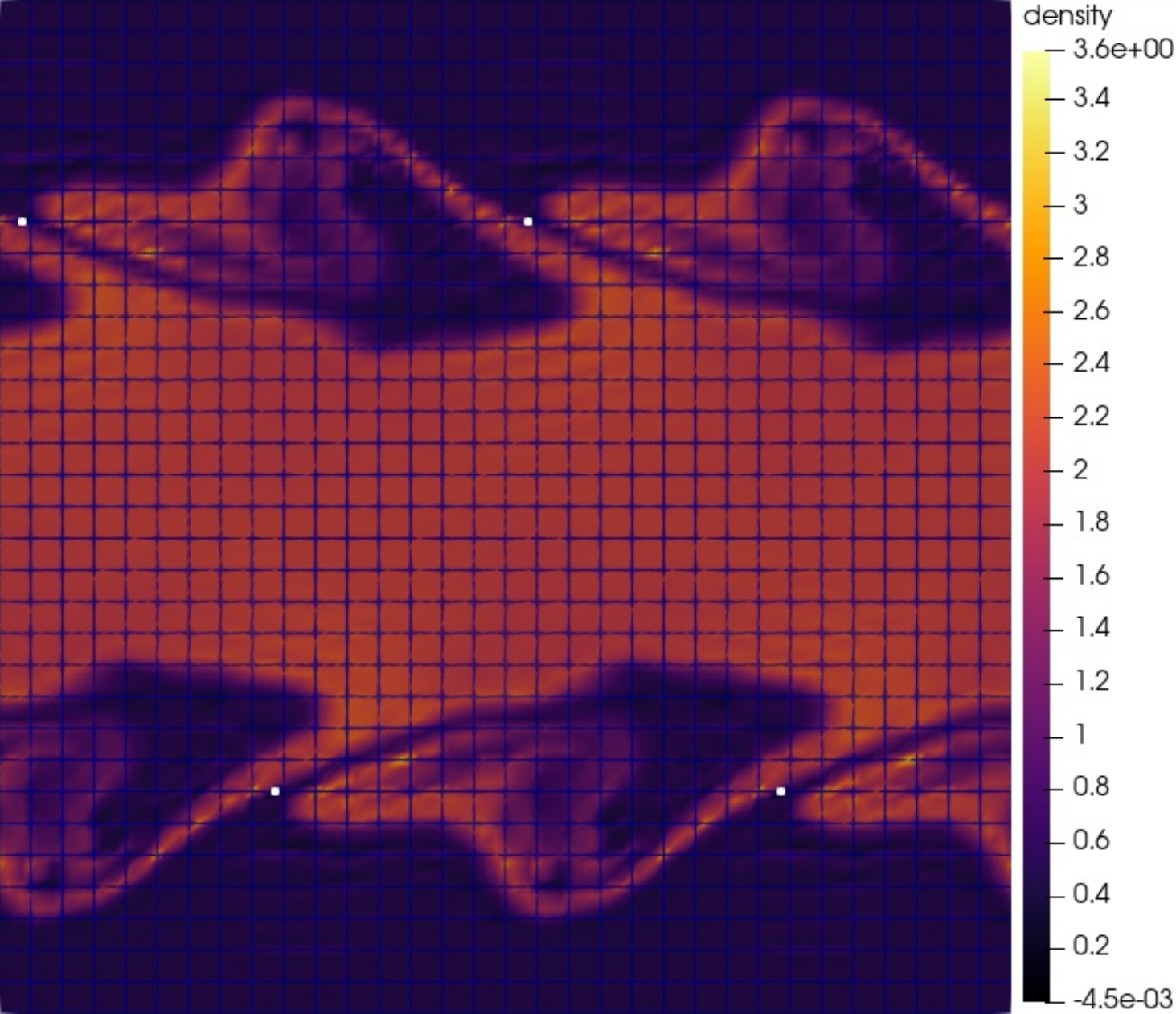}
    		\caption{DGSEM, Ranocha flux, crashed at $t=3.66$}
    		\label{fig:DGSEM_Ranocha_N4_1024elements}
  	\end{subfigure}%
  	\caption{
	Density profile for the DG-USBP method and DGSEM on $J=1024$ elements when the Kelvin--Helmholtz instability simulations crashed.
	Both methods use six LGL nodes per coordinate direction and element.
	Hence, all simulations use the same number of DOFs.
	Furthermore, the DG-USBP methods use $\lambda_5 = -10^{-1}$.
	The white spots mark points where the pressure (DG-USBP) or the density (DGSEM) is non-positive.
  	}
  	\label{fig:KH_N4_1024elements}
\end{figure}

\cref{fig:KH_N4_1024elements} compares the density profiles from simulations of the Kelvin--Helmholtz instability captured at the crash time.
These simulations were conducted on $J=1024$ elements using the DG-USBP method and the DGSEM.
For both the DG-USBP method and DGSEM, we used four ($N=4$) LGL nodes per coordinate direction and element.
That is, all simulations were performed with the same number of DOFs, ensuring a consistent comparison base.
Furthermore, we used $\lambda_4 = -10^{-1}$ for the DG-USBP method and the volume fluxes of Shima et al.\ and Ranocha for the DGSEM.
In \cref{fig:KH_N4_1024elements}, the white spots indicate locations where non-positive values are observed for pressure (in the case of DG-USBP) or density (for DGSEM).
Notably, these problematic nodes consistently occur at the interfaces between elements.
Interestingly, the same was reported for non-periodic FD-USBP operators in \cite{ranocha2023high}, suggesting a general trend.
We also observed this phenomenon for other simulations with different numbers of elements and nodes per element.

\subsection{The inviscid Taylor--Green vortex}
\label{sub:TG_vortex}

Recall that previous works~\cite{mattsson2017diagonal,mattsson2018compatible,lundgren2020efficient,stiernstrom2021residual,ranocha2023high,duru2024dual} have demonstrated that combining high-order USBP operators with flux vector splittings can enhance the robustness of FD schemes without introducing excessive artificial dissipation.
In this work, we have investigated whether combining USBP operators with flux vector splittings can be applied to nodal DG methods with similar success.
However, the numerical tests presented above demonstrated both advantages and disadvantages for DG-USBP schemes compared to existing high-order DGSEM methods.
On the one hand, we observed that DG-USBP methods do not necessarily introduce excessive artificial dissipation (\Cref{sub:convergence_advection,sub:convergence_Euler}), yield local linear/energy stability in a setting where existing high-order entropy-stable flux differencing DGSEM produce unstable semi-discretizations (\Cref{sub:local_syability}), and extend to curvilinear while ensuring FSP (\Cref{sub:fsp_tests}).
On the other hand, we found that DG-USBP schemes are often less accurate than existing high-order DGSEM (\Cref{sub:convergence_Euler,sub:isentropic_vortex}), generally yield stiffer semi-discretizations (\Cref{sub:spectral}), and do not always offer improvements in robustness (\Cref{sub:KH_instability})---which contrasts with observations for FD schemes.
Specifically, in simulating the Kelvin--Helmholtz instability in \Cref{sub:KH_instability}, we observed that USBP operators improve robustness for lower-order operators with three LGL points.
Still, this improvement vanishes when higher-order operators with four and five LGL points are used.
Recall that USBP operators add artificial dissipation only to the unresolved modes.
We thus suspect that the reduced improvement can be explained by the ratio $1/(N-1)$ between the number of unresolved modes\footnote{The DG-USBP operators considered here add artificial dissipation only to the highest mode}, which is $1$, and the number of resolved modes, which is $N-1$, decreasing as more LGL nodes are used.

To shed additional light on these observations, we investigate the kinetic energy dissipation properties of the DG-USBP method, following the approach of \cite{gassner2016split}. 
To this end, we consider the inviscid Taylor--Green vortex test case for the three-dimensional compressible Euler equations of an ideal gas.
The corresponding initial condition is
\begin{equation}
\begin{aligned}
  	\rho & = 1,
  	\quad
  	p = \frac{\rho}{\mathrm{Ma}^2 \gamma} + \rho \frac{\cos(2 x_1) \cos(2 x_3) + 2 \cos(2 x_2) + 2 \cos(2 x_1) + \cos(2 x_2) \cos(2 x_3)}{16}, \\
  v_1 & = \sin(x_1) \cos(x_2) \cos(x_3),
  \quad
  v_2 = -\cos(x_1) \sin(x_2) \cos(x_3), \quad
  v_3 = 0,
\end{aligned}
\end{equation}
with Mach number $\mathrm{Ma} = 0.1$.
The domain is $[-\pi,\pi]^3$ with periodic boundary conditions and the time interval is $[0, 20]$. 
Time integration is performed as in \Cref{sub:KH_instability}.

\begin{figure}[tb]
	\centering
	\begin{subfigure}[b]{0.45\textwidth}
		\includegraphics[width=\textwidth]{%
      		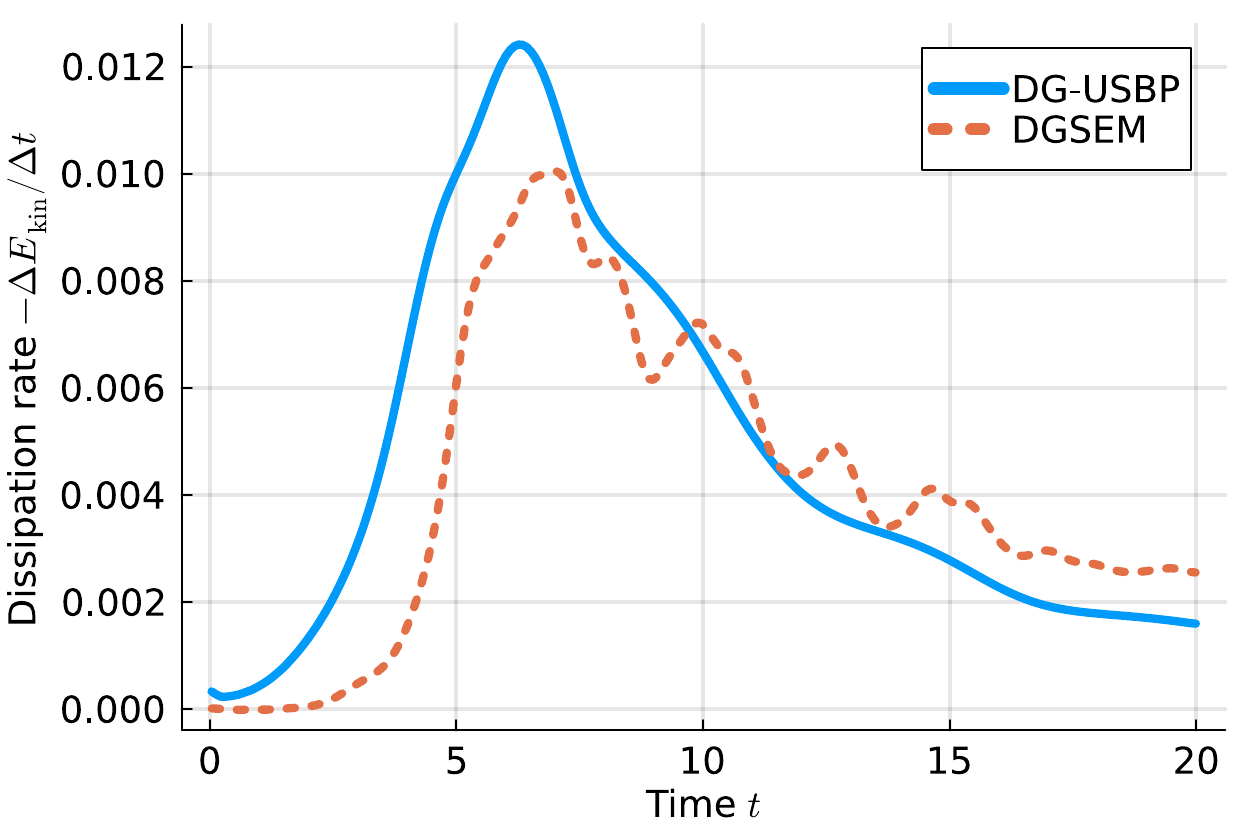}
    		\caption{Dissipation rate for $N=3$}
    		\label{fig:TGV_dissipation_rate_level4_N3_lambda_em1}
  	\end{subfigure}%
	~
	\begin{subfigure}[b]{0.45\textwidth}
		\includegraphics[width=\textwidth]{%
      		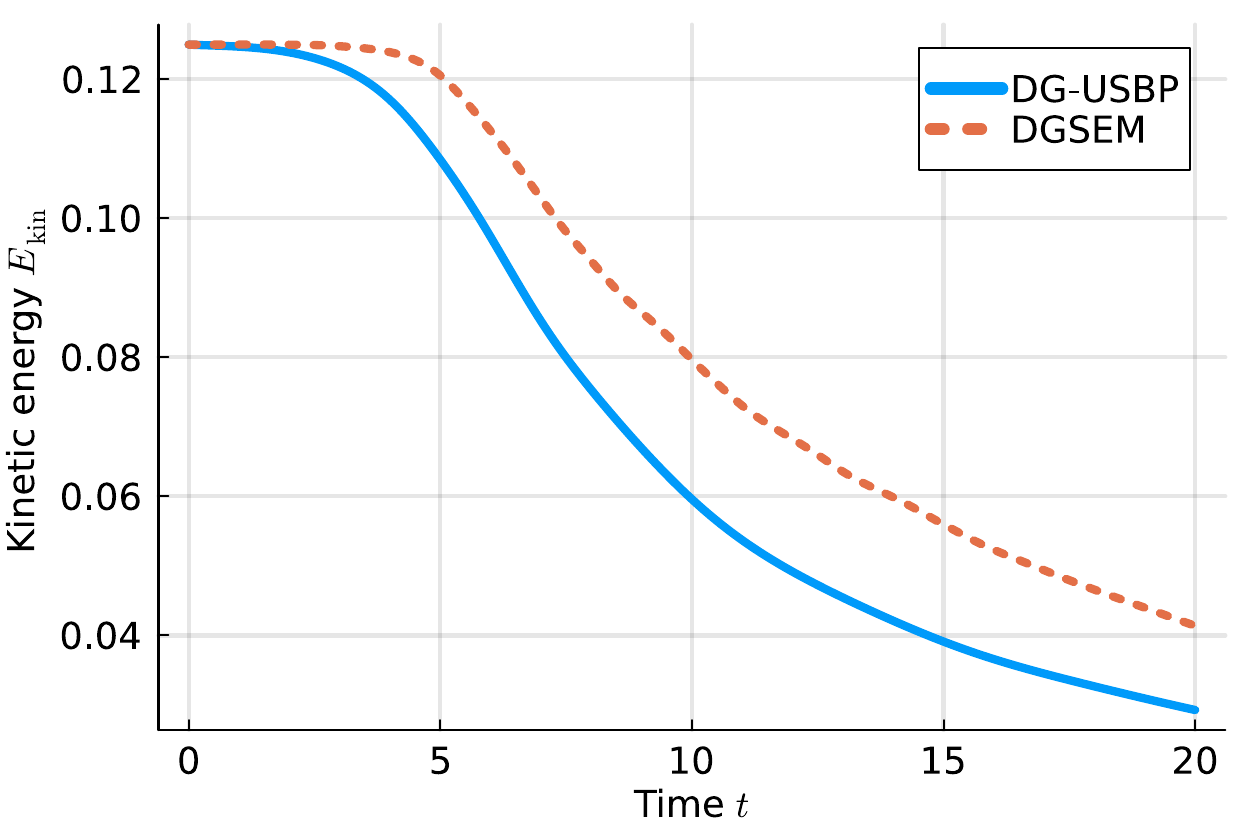}
    		\caption{Kinetic energy for $N=3$}
    		\label{fig:TGV_kinetic_energy_level4_N3_lambda_em1}
  	\end{subfigure}%
	\\
	\begin{subfigure}[b]{0.45\textwidth}
		\includegraphics[width=\textwidth]{%
      		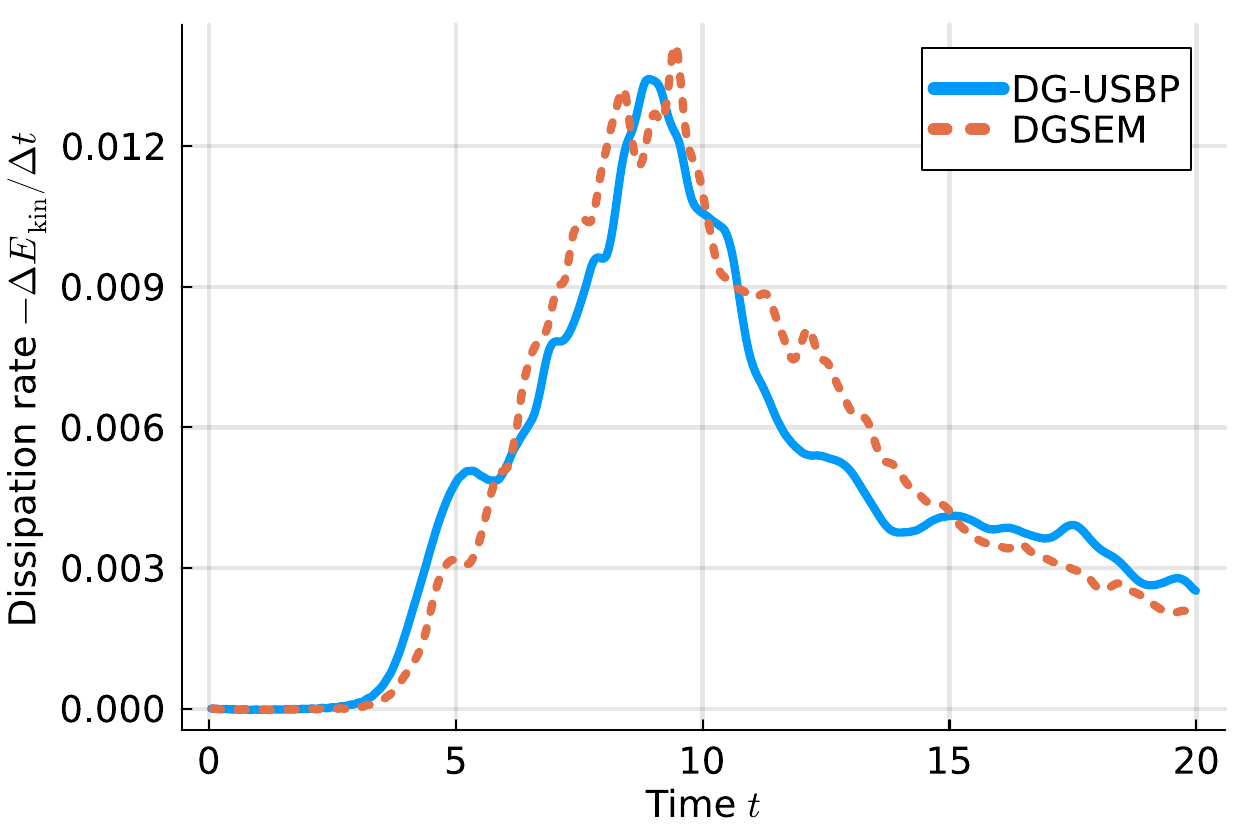}
    		\caption{Dissipation rate for $N=4$}
    		\label{fig:TGV_dissipation_rate_level4_N4_lambda_em1}
  	\end{subfigure}%
	~
	\begin{subfigure}[b]{0.45\textwidth}
		\includegraphics[width=\textwidth]{%
      		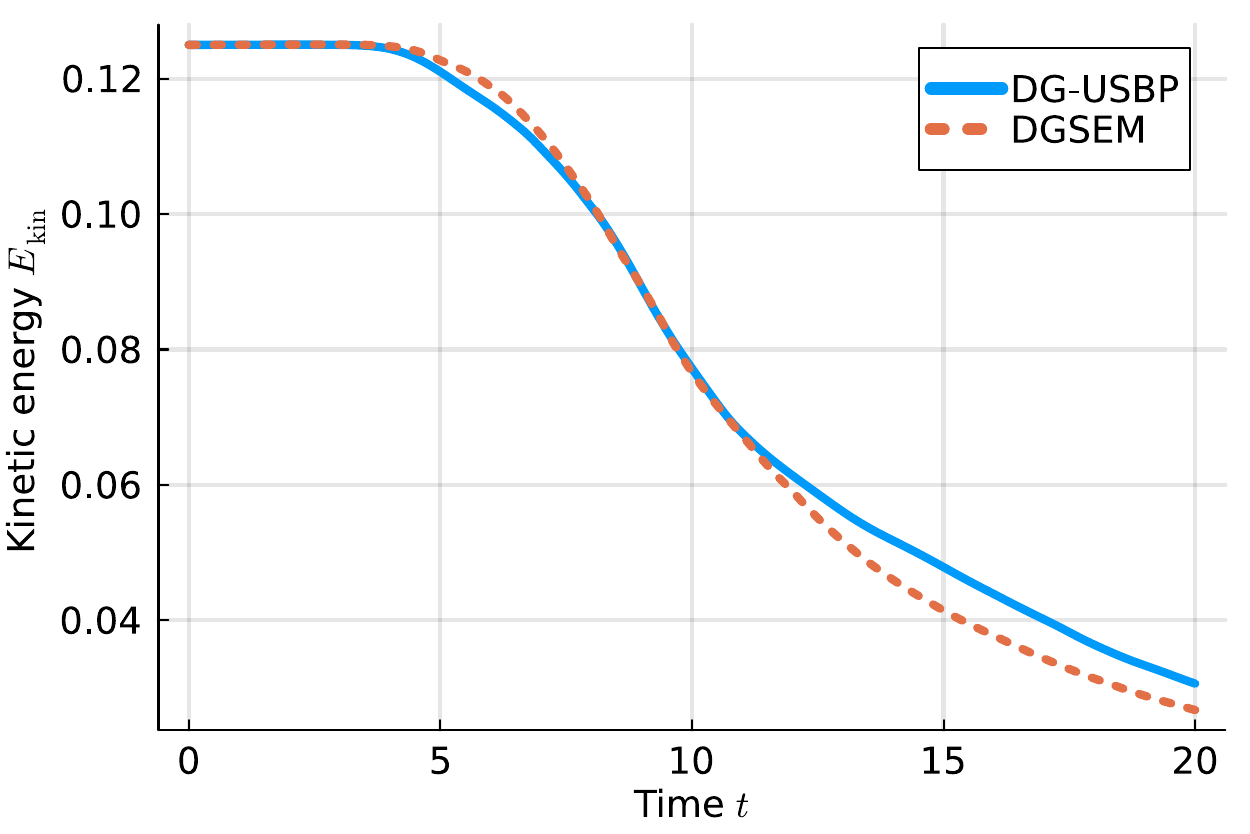}
    		\caption{Kinetic energy for $N=4$}
    		\label{fig:TGV_kinetic_energy_level4_N4_lambda_em1}
  	\end{subfigure}%
  	\caption{
		Discrete kinetic energy and dissipation rate for the inviscid Taylor--Green vortex.
		We compare the results of a DG-USBP method with the DGSEM.
		Both methods use $J=4096$ elements with three ($N=3$, top row) and four ($N=4$, bottom row) LGL points per element and coordinate direction.
		The DGSEM uses the entropy-conservative volume flux of Ranocha and the local Lax-Friedrichs (Rusanov) surface flux.
		The DG-USBP method uses the Steger--Warming splitting and $\lambda_{N} = -10^{-1}$. 
  	}
  	\label{fig:TG_dissipation}
\end{figure}

\cref{fig:TG_dissipation} depicts the discrete kinetic energies and their dissipation rates of a DG-USBP method with a DGSEM.
Both methods use $J=4096$ elements with three ($N=3$, \cref{fig:TGV_dissipation_rate_level4_N3_lambda_em1,fig:TGV_kinetic_energy_level4_N3_lambda_em1}) and four ($N=4$, \cref{fig:TGV_dissipation_rate_level4_N4_lambda_em1,fig:TGV_kinetic_energy_level4_N4_lambda_em1}) LGL points per element and coordinate direction.
The DGSEM uses the entropy-conservative volume flux of Ranocha and the local Lax-Friedrichs (Rusanov) surface flux.
The DG-USBP method uses the Steger--Warming splitting and $\lambda_{N} = -10^{-1}$.
The discrete version of the total kinetic energy is computed as
\begin{equation}
	E_\mathrm{kin}(t)
		= \int \frac{1}{2} \rho(t, x) v(t, x)^2 \dif x
\end{equation}
using the quadrature rule associated with the SBP norm matrix $P$ at every tenth accepted time step.
Central finite differences are then employed to calculate the discrete kinetic energy dissipation rate, $-\Delta E_\mathrm{kin} / \Delta t$, approximating $-\dif E_\mathrm{kin} / \dif t$.
Notably, in the case of three LGL points ($N=3$, \cref{fig:TGV_dissipation_rate_level4_N3_lambda_em1,fig:TGV_kinetic_energy_level4_N3_lambda_em1}), we observe the DG-USBP method often has a higher dissipation rate than the DGSEM and consequently to have smaller kinetic energies.
This implies that the DG-USBP method introduces more artificial dissipation, which also explains the improved robustness we have observed in \Cref{sub:KH_instability} when using three LGL points.
At the same time, for four LGL points ($N=4$, \cref{fig:TGV_dissipation_rate_level4_N4_lambda_em1,fig:TGV_kinetic_energy_level4_N4_lambda_em1}), we find the kinetic energy of the DG-USBP method to be larger than that of the DGSEM towards the end of the simulation.
This implies that now the DG-USBP method introduces less artificial dissipation, which explains the reduced robustness we have observed in \Cref{sub:KH_instability} when using four LGL points.

\subsubsection*{Discussion of the results}

The above results indicate that the DG-USBP method introduces more artificial dissipation than DGSEM for three LGL points and less for four LGL points.
This also explains the varying improvements in robustness we have observed in \Cref{sub:KH_instability}.
Moreover, the above findings provide additional evidence that combining USBP operators with flux vector splittings generally improves the robustness of nodal DG methods less than high-order FD schemes.
The above results indicate that this reduced improvement in robustness can be explained by the decreasing ratio $1/(N-1)$ between the number of unresolved modes, which is $1$, and the number of resolved modes, which is $N-1$, decreasing as more LGL nodes are used.
\section{Concluding remarks} 
\label{sec:summary} 

Previous studies, such as \cite{ranocha2023high,duru2024dual}, have demonstrated that combining high-order USBP operators with flux vector splittings can increase the robustness of FD schemes for under-resolved simulations without introducing excessive artificial dissipation. 
Here, we investigated whether combining USBP operators with flux vector splittings can be applied to nodal DG methods with similar success. 
We began by demonstrating the existence of USBP operators on arbitrary grid points and providing a straightforward procedure for their construction. 
While previous works focused on diagonal-norm FD-USBP operators, our discussion encompassed a broader class of USBP operators. 
Importantly, the proposed construction procedure is applied to arbitrary grid points and is not limited to equidistant spacing. 
This generalization enabled us to develop novel USBP operators on LGL points, which are well-suited for nodal DG methods. 
We then examined the robustness properties of the resulting DG-USBP methods for challenging examples of the compressible Euler equations, such as Kelvin-Helmholtz instabilities. 
Similar to high-order FD-USBP schemes, we found that combining flux vector splitting techniques with DG-USBP operators does not lead to excessive artificial dissipation. 
Furthermore, we demonstrated local linear/energy stability for the DG-USBP semi-discretizations in a setting where existing high-order entropy-stable flux differencing DGSEM produce unstable semi-discretizations. 
However, we also observed that the improvement in robustness offered by USBP operators is less significant for high-order DG methods compared to FD schemes. 
Specifically, the increase in robustness diminished as more LGL points were used per element and coordinate direction. 
We suspected this reduced improvement could be attributed to USBP methods adding dissipation only to unresolved modes. 
After all, FD schemes typically have more unresolved modes than nodal DG methods. 
Still, the DG-USBP approach presents an alternative to entropy stable flux differencing DGSEM. 
It is potentially more convenient for tackling complex problems where a reliable flux splitting method, like the Lax--Friedrichs one, is readily available while conducting an entropy analysis becomes particularly challenging.

While the current iteration of DG-USBP methods has limitations, there is room for improvement. 
Future endeavors could explore the potential of optimally selecting these operators' dissipation parameters in a manner similar to \cite{discacciati2020controlling,zeifang2021data,hillebrand2023applications}. 
In this study, we focused on establishing the general theory, providing a construction methodology, and demonstrating the practical application of USBP operators in DG methods rather than fine-tuning the operators. 
Additionally, although not explicitly demonstrated here, the established general theory and construction framework for USBP operators extend to multi-dimensional SBP operators \cite{nordstrom2001finite,hicken2016multidimensional} and FSBP operators for general function spaces \cite{glaubitz2022summation,glaubitz2023multi,glaubitz2023summation}. 
These avenues will form the core of future research efforts, 
as said methods typically use oversampling and, therefore, have more unresolved modes that can be dissipated using USBP operators. Thus, we expect USBP operators may offer greater improvements in such cases, similar to what was observed for FD schemes in \cite{ranocha2023high,duru2024dual}.

\section*{Acknowledgements}
JG was supported by the US DOD (ONR MURI) grant \#N00014-20-1-2595. 
HR was supported by the Deutsche Forschungsgemeinschaft (DFG, German Research Foundation, project numbers 513301895 and 528753982 as well as within the DFG priority program SPP~2410 with project number 526031774) and the Daimler und Benz Stiftung (Daimler and Benz foundation,
project number 32-10/22).
ARW was funded through Vetenskapsr{\aa}det, Sweden grant
agreement 2020-03642 VR.
MSL received funding through the DFG research unit FOR~5409 "Structure-Preserving Numerical Methods
for Bulk- and Interface Coupling of Heterogeneous Models (SNuBIC)" (project number 463312734),
as well as through a DFG individual grant (project number 528753982).
P\"O was supported by the DFG within the priority research program
SPP 2410, project OE 661/5-1
(525866748) and under the personal grant 520756621 (OE 661/4-1).
GG acknowledges funding through the Klaus-Tschira Stiftung via the project ``HiFiLab'' and received funding through the DFG research unit FOR 5409 ``SNUBIC'' and through the BMBF funded project ``ADAPTEX''.
Some of the computations were enabled by resources provided by the National Academic Infrastructure for Supercomputing in Sweden (NAISS), partially funded by the Swedish Research Council through grant agreement no.\ 2022-06725.

\appendix
\section{Proof of \cref{thm:existence}} 
\label{app:proof_existence} 

Henceforth, let $\mathbf{x} = [x_1,\dots,x_N]^T$ be a grid on the computational domain $[x_L,x_R]$ and $\mathcal{P}_d$ be the space of polynomials up to degree $d$ with $d+1 \leq N$. 
We first demonstrate part (b) of \cref{thm:existence}. 

\begin{lemma}[Part (b) of \cref{thm:existence}]\label{lem:2nd_statement}
	Let $D_{\pm} = P^{-1}( Q_{\pm} + B/2 )$ be degree $d$ USBP operators with $Q_+ + Q_+^T = S$. 
	Then, $D = (D_+ + D_-)/2$ is a degree $d$ SBP operator and $S$ satisfies $S \mathbf{f} = \mathbf{0}$ for all $f \in \mathcal{P}_d$.  
\end{lemma}

\begin{proof} 
	The first assertion, $D = (D_+ + D_-)/2$ being a degree $d$ SBP operator, follows immediately from \cref{lem:connection}. 
	The second assertion, $S$ satisfying $S \mathbf{f} = \mathbf{0}$ for all $f \in \mathcal{P}_d$, has implicitly already been noted in \cite{mattsson2017diagonal,linders2020properties,linders2022eigenvalue}.
	For completeness, we briefly revisit the argument. 
	The relation \cref{eq:relation_D_pm} between $D_-$ and $D_+$ implies 
	\begin{equation}\label{eq:relation_D_pm_aux} 
	D_+ \mathbf{f} - D_- \mathbf{f} = P^{-1}S \mathbf{f}.
\end{equation}
	Furthermore, for $f \in \mathcal{P}_d$, (i) in \cref{def:USBP} yields $D_+ \mathbf{f} - D_- \mathbf{f} = \mathbf{0}$. 
	In this case, \cref{eq:relation_D_pm_aux} becomes $P^{-1}S \mathbf{f} = \mathbf{0}$, which shows the assertion. 
\end{proof}

It remains to prove statement (a) of \cref{thm:existence}. 

\begin{lemma}[Part (a) of \cref{thm:existence}]\label{lem:1st_statement}  
	Let $D = P^{-1}( Q + B/2 )$ be a degree $d$ USBP operator and let $S$ be a symmetric and negative semi-definite matrix with $S \mathbf{f} = \mathbf{0}$ for all $f \in \mathcal{P}_d$. 
	Then, $D_{\pm} = P^{-1}( Q_{\pm} + B/2 )$ with $Q_{\pm} = Q \pm S/2$ are degree $d$ USBP operators with $Q_+ + Q_+^T = S$.
\end{lemma} 

\begin{proof} 
	We must show that $D_{\pm} = P^{-1}( Q_{\pm} + B/2 )$ satisfy (i), (iv), and (v) in \cref{def:USBP}. 
	We first demonstrate (i) in \cref{def:USBP}. 
	Let $f \in \mathcal{P}_d$. 
	Since $D_{\pm} = P^{-1}( Q_{\pm} + B/2 )$ and $S \mathbf{f} = \mathbf{0}$, we have 
	\begin{equation} 
	\begin{aligned}
		D_+ \mathbf{f} 
			& = D \mathbf{f} + P^{-1} S \mathbf{f} / 2
			= \mathbf{f\,'}, \\ 
		D_- \mathbf{f} 
			& = D \mathbf{f} - P^{-1} S \mathbf{f} / 2 
			= \mathbf{f\,'}, 
	\end{aligned}
	\end{equation}  
	which shows that the exactness conditions (i) in \cref{def:USBP} are satisfied. 
 	Next, (iv) in \cref{def:USBP} results from 
	\begin{equation} 
		Q_+ + Q_-^T 
			= (Q + S/2) + (Q - S/2)^T 
			= 0, 
	\end{equation} 
	where the last equation follows from $S$ being symmetric and $Q$ being anti-symmetric (see condition (iv) in \cref{def:SBP}). 
	Finally, condition (v) in \cref{def:USBP} holds since 
	\begin{equation} 
		Q_+ + Q_+^T 
			= (Q + S/2) + (Q + S/2)^T 
			= S, 
	\end{equation} 
	where the last equation follows from $S$ being symmetric and $Q$ being anti-symmetric. 
	This completes the proof. 
\end{proof}

\cref{thm:existence} follows from combining \cref{lem:2nd_statement,lem:1st_statement}.
\section{Diagonal-norm USBP operators on LGL nodes} 
\label{app:operators} 
\allowdisplaybreaks

Below, the DG-USBP operators on four ($N=4$), five ($N=5$), and six ($N=6$) LGL nodes on $\Omega_{\rm ref} = [-1,1]$ are presented. 
The operators are defined by the LGL nodes $\mathbf{x} = [x_1,\dots,x_N]^T$, the weights $\mathbf{p} = [p_1,\dots,p_N]^T$ associated with the diagonal norm matrix $P = \diag(\mathbf{p})$, the central degree $N-1$ SBP operator $D$, and the symmetric negative semi-definite dissipation matrix $S$. 
Recall that we get the USBP operators as $D_{\pm} = D \pm P^{-1} S /2$. 
Furthermore, the dissipation matrix $S$ can be decomposed as $S = V \Lambda V^T$ with $\Lambda = \diag(\lambda_1,\dots,\lambda_N)$ and $V$ being the Vandermonde matrix of a basis of DOPs up to degree $N-1$.

\subsection{DG-USBP operators on four LGL nodes} 
\label{app:operators_LGL_4nodes} 

The four ($N=4$) LGL nodes and associate weights on $\Omega_{\rm ref} = [-1,1]$ are $x_1 = -1$, $x_2 = -1/\sqrt{5}$, $x_3 = -x_2$, $x_4 = -x_2$ and $p_1 = 1/6$, $p_2 = 5/6$, $p_3 = p_2$, $p_4 = p_1$.
The central SBP operator $D = [D_{n,m}]_{n,m=1}^N$ has the entries 
{\small
\begin{align*}
	& D_{1,1} = -3.000000000000000000, \quad
	&& D_{1,2} = 4.045084971874736368, \quad
	&& D_{1,3} = -1.545084971874737478, \\
	& D_{1,4} = 0.500000000000000000, \quad
	&& D_{2,1} = -0.809016994374947229, \quad
	&& D_{2,2} = -0.000000000000000333, \\
	& D_{2,3} = 1.118033988749894903, \quad
	&& D_{2,4} = -0.309016994374947451, \quad
	&& D_{3,1} = 0.309016994374947451, \\
	& D_{3,2} = -1.118033988749894903, \quad
	&& D_{3,3} = 0.000000000000000222, \quad
	&& D_{3,4} = 0.809016994374947340, \\
	& D_{4,1} = -0.500000000000000000, \quad
	&& D_{4,2} = 1.545084971874737256, \quad
	&& D_{4,3} = -4.045084971874737256, \\
	& D_{4,4} = 3.000000000000000000.
\end{align*}
}
The Vandermonde matrix $V = [V_{n,m}]_{n,m=1}^N$ of a basis of DOPs up to degree $N-1$ is given by 
{\small
\begin{align*}
	& V_{1,1} = 0.500000000000000000, \quad
	&& V_{1,2} = -0.645497224367902800, \quad
	&& V_{1,3} = 0.499999999999999889, \\
	& V_{1,4} = -0.288675134594812921, \quad
	&& V_{2,1} = 0.500000000000000000, \quad
	&& V_{2,2} = -0.288675134594812866, \\
	& V_{2,3} = -0.500000000000000000, \quad
	&& V_{2,4} = 0.645497224367902800, \quad
	&& V_{3,1} = 0.500000000000000000, \\
	& V_{3,2} = 0.288675134594812866, \quad
	&& V_{3,3} = -0.500000000000000000, \quad
	&& V_{3,4} = -0.645497224367902800, \\
	& V_{4,1} = 0.500000000000000000, \quad
	&& V_{4,2} = 0.645497224367902800, \quad
	&& V_{4,3} = 0.499999999999999889, \\
	& V_{4,4} = 0.288675134594812921.
\end{align*}
}
Finally, once $\lambda_1,\dots,\lambda_N$ have been fixed, we get the dissipation matrix as $S = V \Lambda V^T$ and the DG-USBP operators as $D_{\pm} = D \pm P^{-1} S /2$.

\subsection{DG-USBP operators on five LGL nodes} 
\label{app:operators_LGL_5nodes} 

The five ($N=5$) LGL nodes and associate weights on $\Omega_{\rm ref} = [-1,1]$ are $x_1 = -1$, $x_2 = -\sqrt{3/7}$, $x_3 = 0$, $x_4 = -x_2$, $x_5 = -x_1$ and $p_1 = 1/10$, $p_2 = 49/90$, $p_3 = 32/45$, $p_4 = p_2$, $p_5 = p_1$.
The central SBP operator $D = [D_{n,m}]_{n,m=1}^N$ has the entries 
{\small
\begin{align*}
	& D_{1,1} = -5.000000000000000000, \quad
	&& D_{1,2} = 6.756502488724240862, \quad
	&& D_{1,3} = -2.666666666666667407, \\
	& D_{1,4} = 1.410164177942426988, \quad
	&& D_{1,5} = -0.500000000000000000, \quad
	&& D_{2,1} = -1.240990253030982648, \\
	& D_{2,2} = -0.000000000000000444, \quad
	&& D_{2,3} = 1.745743121887939342, \quad
	&& D_{2,4} = -0.763762615825973379, \\
	& D_{2,5} = 0.259009746969017129, \quad
	&& D_{3,1} = 0.374999999999999889, \quad
	&& D_{3,2} = -1.336584577695453246, \\
	& D_{3,3} = 0.000000000000000000, \quad
	&& D_{3,4} = 1.336584577695453024, \quad
	&& D_{3,5} = -0.374999999999999833, \\
	& D_{4,1} = -0.259009746969017129, \quad
	&& D_{4,2} = 0.763762615825973268, \quad
	&& D_{4,3} = -1.745743121887939342, \\
	& D_{4,4} = 0.000000000000000444, \quad
	&& D_{4,5} = 1.240990253030982648, \quad
	&& D_{5,1} = 0.500000000000000111, \\
	& D_{5,2} = -1.410164177942426766, \quad
	&& D_{5,3} = 2.666666666666667407, \quad
	&& D_{5,4} = -6.756502488724240862, \\
	& D_{5,5} = 5.000000000000000000, 
\end{align*}
}
The Vandermonde matrix $V = [V_{n,m}]_{n,m=1}^N$ of a basis of DOPs up to degree $N-1$ is given by 
{\small
\begin{align*}
	& V_{1,1} = 0.447213595499957928, \quad
	&& V_{1,2} = -0.591607978309961591, \quad
	&& V_{1,3} = 0.500000000000000000, \\
	& V_{1,4} = -0.387298334620741758, \quad
	&& V_{1,5} = 0.223606797749978908, \quad
	&& V_{2,1} = 0.447213595499957928, \\
	& V_{2,2} = -0.387298334620741702, \quad
	&& V_{2,3} = -0.166666666666666685, \quad
	&& V_{2,4} = 0.591607978309961591, \\
	& V_{2,5} = -0.521749194749950962, \quad
	&& V_{3,1} = 0.447213595499957928, \quad
	&& V_{3,2} = 0.000000000000000000, \\
	& V_{3,3} = -0.666666666666666741, \quad
	&& V_{3,4} = 0.000000000000000000, \quad
	&& V_{3,5} = 0.596284793999943941, \\
	& V_{4,1} = 0.447213595499957928, \quad
	&& V_{4,2} = 0.387298334620741702, \quad
	&& V_{4,3} = -0.166666666666666685, \\
	& V_{4,4} = -0.591607978309961591, \quad
	&& V_{4,5} = -0.521749194749950962, \quad
	&& V_{5,1} = 0.447213595499957928, \\
	& V_{5,2} = 0.591607978309961591, \quad
	&& V_{5,3} = 0.500000000000000000, \quad
	&& V_{5,4} = 0.387298334620741758, \\
	& V_{5,5} = 0.223606797749978908.
\end{align*}
}
Finally, once $\lambda_1,\dots,\lambda_N$ have been fixed, we get the dissipation matrix as $S = V \Lambda V^T$ and the DG-USBP operators as $D_{\pm} = D \pm P^{-1} S /2$.

\subsection{DG-USBP operators on six LGL nodes} 
\label{app:operators_LGL_6nodes} 

The six ($N=6$) LGL nodes and associate weights on $\Omega_{\rm ref} = [-1,1]$ are $x_1 = -1$, $x_2 = -\sqrt{ ( 7+2\sqrt{7} ) / 21 }$, $x_3 = -\sqrt{ ( 7-2\sqrt{7} ) / 21 }$, $x_4 = -x_3$, $x_5 = -x_2$, $x_6 = -x_1$ and $p_1 = 1/15$, $p_2 = (14-\sqrt{7})/30$, $p_3 = (14+\sqrt{7})/30$, $p_4 = p_3$, $p_5 = p_2$, $p_6 = p_1$. 
The central SBP operator $D = [D_{n,m}]_{n,m=1}^N$ has the entries 
{\small
\begin{align*}
	& D_{1,1} = -7.500000000000000000, \quad
    && D_{1,2} = 10.141415936319667424, \quad
    && D_{1,3} = -4.036187270305348740, \\
    & D_{1,4} = 2.244684648176165975, \quad
    && D_{1,5} = -1.349913314190487768, \quad
    && D_{1,6} = 0.500000000000000111, \\
    & D_{2,1} = -1.786364948339095315, \quad
    && D_{2,2} = 0.000000000000001110, \quad
    && D_{2,3} = 2.523426777429455203, \\
    & D_{2,4} = -1.152828158535929015, \quad
    && D_{2,5} = 0.653547507429800167, \quad
    && D_{2,6} = -0.237781177984231401, \\
    & D_{3,1} = 0.484951047853569239, \quad
    && D_{3,2} = -1.721256952830232834, \quad
    && D_{3,3} = 0.000000000000000222, \\
    & D_{3,4} = 1.752961966367866165, \quad
    && D_{3,5} = -0.786356672223240794, \quad
    && D_{3,6} = 0.269700610832039001, \\
    & D_{4,1} = -0.269700610832039001, \quad
    && D_{4,2} = 0.786356672223240682, \quad
    && D_{4,3} = -1.752961966367865942, \\
    & D_{4,4} = 0.000000000000000222, \quad
    && D_{4,5} = 1.721256952830233278, \quad
    && D_{4,6} = -0.484951047853569350, \\
    & D_{5,1} = 0.237781177984231346, \quad
    && D_{5,2} = -0.653547507429800167, \quad
    && D_{5,3} = 1.152828158535929237, \\
    & D_{5,4} = -2.523426777429455203, \quad
    && D_{5,5} = -0.000000000000000888, \quad
    && D_{5,6} = 1.786364948339095537, \\
	& D_{6,1} = -0.500000000000000000, \quad
	&& D_{6,2} = 1.349913314190487768, \quad
	&& D_{6,3} = -2.244684648176165531, \\
	& D_{6,4} = 4.036187270305348740, \quad
	&& D_{6,5} = -10.141415936319669200, \quad
	&& D_{6,6} = 7.500000000000000888.
\end{align*}
}
The Vandermonde matrix $V = [V_{n,m}]_{n,m=1}^N$ of a basis of DOPs up to degree $N-1$ is given by 
{\small
\begin{align*}
    & V_{1,1} = 0.408248290463863073, \quad
    && V_{1,2} = -0.547722557505166074, \quad
    && V_{1,3} = 0.483045891539647831, \\
    & V_{1,4} = -0.408248290463863017, \quad
    && V_{1,5} = 0.316227766016837830, \quad
    && V_{1,6} = -0.182574185835055497, \\
    & V_{2,1} = 0.408248290463863073, \quad
    && V_{2,2} = -0.419038058655589740, \quad
    && V_{2,3} = 0.032338332982759031, \\
    & V_{2,4} = 0.367654222400928044, \quad
    && V_{2,5} = -0.576443896275456780, \quad
    && V_{2,6} = 0.435014342463467985, \\
    & V_{3,1} = 0.408248290463863073, \quad
    && V_{3,2} = -0.156227735687855862, \quad
    && V_{3,3} = -0.515384224522407175, \\
    & V_{3,4} = 0.445155822251155409, \quad
    && V_{3,5} = 0.260216130258618783, \quad
    && V_{3,6} = -0.526715472069829382, \\
    & V_{4,1} = 0.408248290463863073, \quad
    && V_{4,2} = 0.156227735687855862, \quad
    && V_{4,3} = -0.515384224522407175, \\
    & V_{4,4} = -0.445155822251155409, \quad
    && V_{4,5} = 0.260216130258618783, \quad
    && V_{4,6} = 0.526715472069829382, \\
    & V_{5,1} = 0.408248290463863073, \quad
    && V_{5,2} = 0.419038058655589740, \quad
    && V_{5,3} = 0.032338332982759031, \\
    & V_{5,4} = -0.367654222400928044, \quad
	&& V_{5,5} = -0.576443896275456780, \quad
	&& V_{5,6} = -0.435014342463467985, \\
	& V_{6,1} = 0.408248290463863073, \quad
	&& V_{6,2} = 0.547722557505166074, \quad
	&& V_{6,3} = 0.483045891539647831, \\
	& V_{6,4} = 0.408248290463863017, \quad
	&& V_{6,5} = 0.316227766016837830, \quad
	&& V_{6,6} = 0.182574185835055497.
\end{align*}
}
Finally, once $\lambda_1,\dots,\lambda_N$ have been fixed, we get the dissipation matrix as $S = V \Lambda V^T$ and the DG-USBP operators as $D_{\pm} = D \pm P^{-1} S /2$.
\section{More examples of USBP operators} 
\label{app:examples} 

We exemplify the construction of USBP operators on the reference element $\Omega_{\rm ref} = [-1,1]$. 
\emph{
Henceforth, we round all reported numbers to the second decimal place for ease of presentation. 
}

\subsection{USBP operators on Gauss--Legendre points} 
\label{sub:examples_Legendre}

We demonstrate the construction of a degree two ($d=2$) diagonal-norm USBP operator on four ($N=4$) Gauss--Legendre points in $\Omega_{\rm ref} = [-1,1]$. 
Notably, the Gauss--Legendre points do not include the boundary points. 

Let us first address (P1). 
We start from a degree three diagonal-norm SBP operator, obtained by evaluating the derivatives of the corresponding Lagrange basis functions at the Gauss--Legendre grid points.  
The associated points and quadrature weights (diagonal entries of $P$) are $\mathbf{x} \approx \frac{1}{5} [ -43, -17, 17, 43 ]$ and $\mathbf{p} \approx \frac{1}{20} [ 7, 13, 13, 7 ]$, respectively. 
Moreover, the SBP and boundary operator are 
\begin{equation} 
	D \approx 
	\frac{1}{100}
	\begin{bmatrix} 
		333 & 486 & -211 & 58 \\ 
		-76 & -38 & 147 & -33 \\ 
		33 & -147 & 38 & 76 \\ 
		-58 & 211 & -486 & 333
	\end{bmatrix}, 
	\quad 
	B \approx 
	\frac{1}{100}
	\begin{bmatrix} 
		-153 & 81 & -40 & 11 \\ 
		0 & 0 & 0 & 0 \\ 
		0 & 0 & 0 & 0 \\ 
		-11 & 40 & -81 & 153
	\end{bmatrix}.
\end{equation} 
Note that the first and last row of $B$ corresponds to the function values of the Lagrange basis functions at the left and right boundary point, $x_L=-1$ and $x_R=1$, respectively.  

We next address (P2), finding a dissipation matrix $S$ with $S \mathbf{f} = \mathbf{0}$ if $f \in \mathcal{P}_2$ and $\mathbf{f}^T S \mathbf{f} < 0$ otherwise. 
We follow \Cref{sub:DOPs} and robustly construct $S$ using a DOP basis on the four Gauss--Legendre grid points.  
The resulting Vandermonde matrix of this DOP basis is 
\begin{equation} 
	V \approx 
	\begin{bmatrix} 
		1/2 & -33/50 & 1/2 & -13/50 \\ 
		1/2 & -13/50 & -1/2 & 33/50 \\ 
		1/2 & 13/50 & -1/2 & -33/50 \\ 
		1/2 & 33/50 & 1/2 & 13/50
	\end{bmatrix}. 
\end{equation} 
Consequently, we get the dissipation matrix as $S = V \Lambda V^T$ with $\Lambda = \diag(\lambda_1,\lambda_2,\lambda_3,\lambda_4)$. 
We choose the first three eigenvalues as zero to ensure that the USBP operator has degree two. 
Furthermore, we choose $\lambda_4 = -1$, which introduces artificial dissipation to the highest, unresolved, mode.\footnote{
Note that $\lambda_4 = -1$ is an arbitrary choice. 
Any negative value yields an admissible dissipation matrix.
}
The resulting dissipation matrix is 
\begin{equation} 
	S \approx 
	\frac{1}{100}
	\begin{bmatrix} 
		-7 & 17 & -17 & 7 \\ 
		17 & -43 & 43 & -17 \\  
		-17 & 43 & -43 & 17 \\  
		7 & -17 & 17 & -7 
	\end{bmatrix}. 
\end{equation} 

Finally, we get degree two diagonal-norm USBP operators on $[-1,1]$ as $D_{\pm} = D \pm P^{-1}S/2$, which yields 
\begin{equation} 
\begin{aligned}
	D_+ & \approx 
	\frac{1}{100}
	\begin{bmatrix} 
		-343 & 511 & -235 & 68 \\ 
		-63 & -72 & 180 & -46 \\
		20 & -114 & 5 & 89 \\
		-48 & 186 & -461 & 81 
	\end{bmatrix}, \\  
	D_- & \approx 
	\frac{1}{100}
	\begin{bmatrix} 
		-324 & 461 & -186 & 48 \\ 
		-89 & -5 & 114 & -20 \\
		46 & -180 & 72 & 63 \\
		-68 & 235 & -511 & 343 
	\end{bmatrix}. 
\end{aligned}
\end{equation}
The norm and boundary matrix of the degree two USBP operators are the same as those of the degree three SBP operator.

\subsection{Dense-norm USBP operators on equidistant points} 
\label{sub:examples_dense} 

We exemplify our construction procedure for USBP operators using a dense-norm (the norm matrix $P$ is non-diagonal) SBP operator of degree three ($d=3$) on four ($N=4$) equidistant points on $[0,1]$. 
The degree three dense-norm SBP operator we start from was derived in \cite[Section 4.2]{fernandez2014generalized} and has the following norm and derivative matrices: 
\begin{equation} 
	P \approx 
	\frac{1}{8}
	\begin{bmatrix} 
		2 & 1 & 0 & 0 \\ 
		1 & 10 & -2 & 0 \\  
		0 & -2 & 10 & 1 \\  
		0 & 0 & 1 & 2 
	\end{bmatrix}, \quad 
	D \approx 
	\frac{1}{6}
	\begin{bmatrix} 
		-11 & 18 & -9 & 2 \\ 
		-2 & -3 & 6 & -1 \\  
		1 & -6 & 3 & 2 \\  
		-2 & 9 & -18 & 11 
	\end{bmatrix}.
\end{equation} 

Let us address (P2), finding a dissipation matrix $S$ with $S \mathbf{f} = \mathbf{0}$ if $f \in \mathcal{P}_2$ and $\mathbf{f}^T S \mathbf{f} < 0$ otherwise. 
Once more, we follow \cref{sub:DOPs} and robustly construct $S$ using a DOP basis on the four equidistant grid points.  
The resulting Vandermonde matrix of this DOP basis is
\begin{equation} 
	V \approx 
	\begin{bmatrix} 
		1/2 & -67/100 & 1/2 & -22/100 \\ 
		1/2 & -22/100 & -1/2 & 67/100 \\ 
		1/2 & 22/100 & -1/2 & -67/100 \\ 
		1/2 & 67/100 & 1/2 & 22/100
	\end{bmatrix}. 
\end{equation} 
Consequently, we get the dissipation matrix as $S = V \Lambda V^T$ with $\Lambda = \diag(\lambda_1,\lambda_2,\lambda_3,\lambda_4)$. 
We choose the first three eigenvalues as zero to get USBP operators with degree two ($d=2$). 
Furthermore, we choose $\lambda_4 = -1$, which introduces artificial dissipation to the highest, unresolved mode. 
The resulting dissipation matrix is 
\begin{equation} 
	S \approx 
	\frac{1}{20}
	\begin{bmatrix} 
		-1 & 3 & -3 & 1 \\ 
		3 & -9 & 9 & -3 \\  
		-3 & 9 & -9 & 3 \\  
		1 & -3 & 3 & -1 
	\end{bmatrix}. 
\end{equation} 

Finally, we get degree two dense-norm USBP operators on four equidistant points on $[0,1]$ as $D_{\pm} = D \pm P^{-1}S/2$, which yields 
\begin{equation} 
\begin{aligned}
	D_+ & \approx 
	\frac{1}{100}
	\begin{bmatrix} 
		-189 & 317 & -167 & 39 \\ 
		-30 & -59 & 109 & -20 \\
		14 & -91 & 41 & 36 \\
		-28 & 133 & -283 & 178 
	\end{bmatrix}, \\  
	D_- & \approx 
	\frac{1}{100}
	\begin{bmatrix} 
		-178 & 283 & -133 & 28 \\
		-36 & -41 & 91 & -14 \\
		20 & -109 & 59 & 30 \\
		-39 & 167 & -317 & 189 
	\end{bmatrix}. 
\end{aligned}
\end{equation}
The norm and boundary matrix of the degree two USBP operators are the same as those of the degree three SBP operator. 

\small
\bibliographystyle{siamplain}
\bibliography{references}

\end{document}